\documentclass[11pt, reqno]{amsart}
\usepackage[latin9]{inputenc}
\usepackage{enumitem}
\usepackage{amsthm}
\usepackage{amstext}
\usepackage{amssymb}
\usepackage{esint}
\usepackage{dsfont}
\usepackage{xcolor}
\usepackage{hyperref}
\usepackage[backend=biber,
doi=false,
url=false,
eprint=false,
isbn=false]{biblatex}
\makeatletter
\numberwithin{equation}{section}
\numberwithin{figure}{section}
\usepackage{enumitem}		
\theoremstyle{plain}

\theoremstyle{plain}

\theoremstyle{plain}

\theoremstyle{plain}

\theoremstyle{definition}

\theoremstyle{remark}

\theoremstyle{remark}
\newtheorem*{rem*}{\protect\remarkname}
\newlist{casenv}{enumerate}{4}
\setlist[casenv]{leftmargin=*,align=left,widest={iiii}}
\setlist[casenv,1]{label={{\itshape\ \casename} \arabic*.},ref=\arabic*}
\setlist[casenv,2]{label={{\itshape\ \casename} \roman*.},ref=\roman*}
\setlist[casenv,3]{label={{\itshape\ \casename\ \alph*.}},ref=\alph*}
\setlist[casenv,4]{label={{\itshape\ \casename} \arabic*.},ref=\arabic*}
\theoremstyle{plain}

\usepackage{amsthm}\usepackage{amsfonts}\usepackage{latexsym}

\newtheorem{theorem}{Theorem}[section]
\newtheorem{lemma}[theorem]{Lemma}
\newtheorem{proposition}[theorem]{Proposition}
\newtheorem{definition}[theorem]{Definition}
\newtheorem{corollary}[theorem]{Corollary}

\newtheorem{example}[theorem]{Example}

\newtheorem{remark}[theorem]{Remark}
\newtheorem*{recalledthm}{Theorem}

\newcommand{\R}{{\mathbb{R}}}
\newcommand{\C}{{\mathbb{C}}}
\newcommand{\F}{\ensuremath{\mathbb{F}}}

\newcommand{\T}{\mathbb{T}}

\newcommand{\Acal}{\mathcal{A}}
\newcommand{\Bcal}{\mathcal{B}}

\newcommand{\Hcal}{\mathcal{H}}

\newcommand{\Ocal}{\mathcal{O}}

\newcommand{\Afrak}{\mathfrak{A}}
\newcommand{\Bfrak}{\mathfrak{B}}
\newcommand{\Cfrak}{\mathfrak{C}}
\newcommand{\Hfrak}{\mathfrak{H}}

\newcommand{\ff}{F} 
\newcommand{\gf}{G}
\newcommand{\hf}{H}

\DeclareMathOperator{\Cheeger}{\Cfrak}

\newcommand{\dmu}{\,\mathrm{d}\mu}
\newcommand{\dnu}{\,\mathrm{d}\nu}
\newcommand{\dx}{\,\mathrm{d}x}
\newcommand{\dy}{\,\mathrm{d}y}
\newcommand{\dxi}{\,\mathrm{d}\xi}

\let\emptyset\varnothing

\DeclareMathOperator{\ran}{ran}

\newcommand{\norm}[1]{\ensuremath{\left\lVert\text{}#1\text{}\right\rVert}}
\newcommand{\ip}[1]{\ensuremath{\left<\text{} \, #1 \, \text{}\right>}}

\DeclareMathOperator{\re}{Re}

\DeclareMathOperator{\id}{id}

\DeclareMathOperator{\GL}{GL}
\DeclareMathOperator{\supp}{supp}
\DeclareMathOperator{\const}{const}

\DeclareMathOperator*{\inter}{int}

\newcommand{\Lcal}{\mathcal{L}}

\newcommand{\Scal}{\mathcal{S}}

\providecommand{\assumptionname}{Assumption}
\providecommand{\corollaryname}{Corollary}
\providecommand{\definitionname}{Definition}
\providecommand{\lemmaname}{Lemma}
\providecommand{\remarkname}{Remark}
\providecommand{\casename}{Case}
\providecommand{\theoremname}{Theorem}
\providecommand{\propositionname}{Proposition}

\makeatother

\providecommand{\assumptionname}{Assumption}
\providecommand{\corollaryname}{Corollary}
\providecommand{\definitionname}{Definition}
\providecommand{\lemmaname}{Lemma}
\providecommand{\remarkname}{Remark}
\providecommand{\casename}{Case}
\providecommand{\theoremname}{Theorem}

\renewbibmacro*{doi+eprint+url}{%
	\iftoggle{bbx:doi}{\printfield{doi}}{}%
	\newunit\newblock
	\iftoggle{bbx:eprint}{\usebibmacro{eprint}}{}%
	\newunit\newblock
	\iftoggle{bbx:url}{%
		\iffieldundef{doi}{\usebibmacro{url+urldate}}{}%
	}{}%
}

\renewbibmacro*{in:}{%
	\ifentrytype{article}
	{}
	{\printtext{\bibstring{in}\intitlepunct}}}

\begin{document}
	
	\title[Cheeger Bounds for Stable Phase Retrieval in RKHSs]{Cheeger Bounds for Stable Phase Retrieval in Reproducing Kernel Hilbert Spaces}

	\author{Hartmut F\"{u}hr, Max Getter}
	
	\address{Hartmut F\"{u}hr \\
		Chair for Geometry and Analysis \\
		RWTH Aachen University \\
		D-52062 Aachen \\
		Germany}
	\email{fuehr@mathga.rwth-aachen.de}
	\address{Max Getter \\
		Chair for Geometry and Analysis \\
		RWTH Aachen University \\
		D-52062 Aachen \\
		Germany}
	\email{getter@mathga.rwth-aachen.de}
	\begin{abstract}
		Phase retrieval seeks to reconstruct a signal from phaseless intensity measurements and, in applications where measurements contain errors, demands stable reconstruction. We study local stability of phase retrieval in reproducing kernel Hilbert spaces. Motivated by Grohs-Rathmair's Cheeger-type estimate for Gabor phase retrieval, we introduce a \emph{kernel Cheeger constant} that quantifies connectedness relative to kernel localization. This notion yields a clean stability certificate: we establish a unified lower bound over both real and complex fields, and in the real case also an upper bound, each in terms of the reciprocal kernel Cheeger constant. Our framework treats finite- and infinite-dimensional settings uniformly and covers discrete, semi-discrete, and continuous measurement domains. For generalized wavelet phase retrieval from (semi-)discrete frames, we bound the kernel Cheeger constant by the Cheeger constant of a data-dependent weighted graph. We further characterize phase retrievability for generalized wavelet transforms and derive a simple sufficient criterion for wavelet sign retrieval in arbitrary dimension for transforms associated with irreducibly admissible matrix groups.
	\end{abstract}
	\maketitle
	\global\long\def\with{\,\middle|\,}

	\hyphenation{an-iso-tropic}
	\noindent \textbf{\small Keywords:}{\small {}  Phase Retrieval; Reproducing Kernel Hilbert Spaces; Cheeger Inequality; Wavelet Transform}{\small \par}
	
	\noindent \textbf{\small 2020 Mathematics Subject Classification:}{\small {} 42C15;
		42C40; 46E22; 94A12}{\small \par}

	 \section{Introduction}\label{sec: intro}
	 In physics and engineering, one often encounters the situation that, when measuring an unknown quantity, an essential part of the information one aims to determine is lost. Certain measurement devices can record only the intensity of the (possibly device-transformed) incoming signal. As a consequence, the measurement process entails an unavoidable loss of information about the phase of the original signal. The term \emph{phase retrieval} is commonly used to describe the problem of reconstructing the missing phase information of the input signal. Prominent instances of this problem arise in quantum mechanics, speech recognition, radar, X-ray crystallography, and coherent diffraction imaging \cite{pauli1990allgemeinen,rabiner1993fundamentals,jaming1999phase,miao1999extending}.
	 
	 The diversity of practical applications is reflected in the variety of corresponding mathematical models. From an abstract point of view phase retrieval is the reconstruction of a vector $f$ from its phaseless measurements $|Tf|$, where $T$ is a linear transformation into a function space. In its broadest generality \cite{freeman2024stable,garcia2025existence}, this perspective leads to the notion of a Banach lattice as the natural ambient space for $Tf$ (and hence $|Tf|$). 
	 
	 The modulus nonlinearity of the measurement process introduces natural ambiguities: any unimodular scalar $\alpha$ yields the same phaseless measurements via $|T(\alpha f)|=|Tf|$. Hence, the linear transform $T:\Bcal \to \Bfrak$ over \(\F\), assuming that $T$ is a map between a Banach space $\Bcal$ and a Banach lattice $\Bfrak$ (thinking of an $L^p$-space as a prototype) over the field \(\F\), is said to do phase retrieval if whenever $f,g\in \Bcal$ satisfy $|Tf|=|Tg|$ then $f=\alpha g$ for some unimodular scalar $\alpha\in \F$.
	 
	 We study phase retrieval over both scalar fields \(\F=\R\) and \(\F=\C\). In the real case \(\F=\R\) (so that \(\Bcal\) and \(\Bfrak\) are real spaces) we also speak of \emph{sign retrieval}. Unless explicitly stated otherwise, all results apply to both fields; when a statement depends on the choice of \(\F\), we indicate this in the assumptions.
	 
	 In applications of the phase retrieval problem, the measurements are inherently contaminated by errors. In order for a reconstruction obtained from phase retrieval to be meaningful in practice, it is therefore desirable that the reconstruction be stable as well. The linear transformation $T$ is said to do \emph{stable phase retrieval} if the recovery of signals is Lipschitz: there exists a constant $C>0$ such that for all $f,g\in \Bcal$,
	 \begin{align}\label{def: stable PR for T}
	 	\inf_{|\alpha|=1} \norm{f-\alpha g}_\Bcal \leq C \norm{|Tf|-|Tg|}_\Bfrak.
	 \end{align}
	 The left-hand side of \eqref{def: stable PR for T} is the distance between $f$ and $g$ when measured by the natural quotient metric on $\Bcal /\!\sim$, where $\sim$ is the equivalence relation that filters all trivial ambiguities, i.e., $f\sim g$ if there exists a unimodular scalar $\alpha\in \T:=\{z\in \F:|z|=1\}$ such that $f=\alpha g$. 
	 
	 If $\Bcal$ is finite-dimensional, then $T$ is known to do phase retrieval if and only if it does stable phase retrieval \cite{balan2013stability,bandeira2014saving,balan2015invertibility,cahill2016PR,balan2016lipschitz}. The situation is fundamentally different if $\Bcal$ is infinite-dimensional. Essentially, if $T$ is the analysis operator of a frame or a continuous Banach frame, then $T$ cannot do stable phase retrieval \cite{cahill2016PR,alaifari2017PRcont}. We give a precise formulation of this statement in the context of our setting, see the theorem of Alaifari--Grohs below. Moreover, when one passes to finite-dimensional approximations of the phase retrieval problem, stability may deteriorate severely as the dimension increases \cite{cahill2016PR,alaifari2021gabor}.
	 
	 The linear transformation $T$ is often isometric. In this case, \eqref{def: stable PR for T} has an equivalent reformulation via phase retrieval in the image space $T(\Bcal)\subset \Bfrak$, namely that for all $F,G\in T(\Bcal)$,
	 \begin{align}\label{def: stable PR for T in image space}
	 	\inf_{\alpha\in \T} \norm{\ff-\alpha \gf}_\Bfrak \leq C \norm{|\ff|-|\gf|}_\Bfrak.
	 \end{align}
	 The choice of the scalar field $\F$ plays a critical role for the
	 phase retrieval problem. If $\F=\R$, then there are only two phases,
	 so that the left-hand side of \eqref{def: stable PR for T in image space} reduces to 
	 \[\norm{\ff-\gf}_\Bcal \wedge \norm{\ff+\gf}_\Bcal,\]
	 where we use the notation $a \wedge b:=\min\{a,b\}$ for any $a,b\in\R$. 
	 If $\F=\C$, by contrast, the infimum in \eqref{def: stable PR for T in image space} is taken over all unimodular scalars $\alpha\in \C$ with $|\alpha|=1$, i.e. over the full unit circle in $\C$. In particular, the infimum ranges over an infinite set, but it is attained (and hence may be written as a minimum).
	 
	 We study the phase retrieval problem in \emph{reproducing kernel Hilbert spaces} (RKHS) \cite{paulsen2016introduction}. More specifically\footnote{We refer the reader to Section~\ref{sec: cheeger RKHS} for more details.}, we assume that $\Hcal\subset \F^X$ is an RKHS subspace of $L^2(X):=L^2(X,\Sigma,\mu)$, for some measure space $(X,\Sigma,\mu)$. Now, $\Hcal$ is said to do phase retrieval if whenever $\ff,\gf\in \Hcal$ satisfy $|\ff|=|\gf|$ then $\ff=\alpha \gf$ for some unimodular scalar $\alpha\in \T$. Equivalently, $\Hcal$ does phase retrieval if and only if the nonlinear mapping
	 \[\Acal: \Hcal /\!\sim \:\to L^2(X),\quad [\ff]_\sim \mapsto |\ff|\]
	 is injective. 
	 The following theorem records the main consequences of \cite{alaifari2017PRcont} in our setting.
	 \begin{recalledthm}[{Alaifari--Grohs \cite{alaifari2017PRcont}}]\label{rethm: PR cont unstable}
	 	Let $X$ be a $\sigma$-compact topological space. Suppose that $\Hcal$ is an RKHS subspace of $L^2(X)$ and assume that its reproducing kernel is continuous.
	 	Then the following hold:
	 	\begin{itemize}
	 		\item If $\Acal$ is injective, then $\Acal^{-1}$ is continuous on $\ran(\Acal)$.
	 		\item If $\Hcal$ is infinite-dimensional, then phase retrieval in $\Hcal$ is unstable: for every $\varepsilon>0$, there exist $\ff,\gf\in \Hcal$ with 
	 		\begin{align*}
	 			\norm{|\ff|-|\gf|}_2<\varepsilon \quad \text{but} \quad \inf_{\alpha \in \T} \norm{\ff-\alpha \gf}_2\geq 1.
	 		\end{align*}
	 	\end{itemize}
	 \end{recalledthm}
	 The theorem states that there is no hope for global Lipschitz-stability 
	 of the phase retrieval problem if $\Hcal$ is infinite-dimensional. This motivates the study of \emph{locally} stable recovery. A signal $\ff\in \Hcal$ is said to be \emph{phase retrievable} if whenever $\gf\in \Hcal$ satisfies $|\ff|=|\gf|$ then $\ff=\alpha \gf$ for some $\alpha\in \T$. We say that $\ff$ does \emph{stable phase retrieval} if there exists a constant $C>0$ such that for all $\gf\in \Hcal$,
	 \begin{align}\label{def: locally stable PR}
	 	\inf_{\alpha \in \T} \norm{\ff-\alpha \gf}_2 \leq C \norm{|\ff|-|\gf|}_2.
	 \end{align}
	 We denote the smallest possible constant $C$ that satisfies \eqref{def: locally stable PR} by $C(\ff)\in [1,\infty]$ and call it the \emph{local stability constant} of $\ff$. 
	 We emphasize that continuity of $\Acal^{-1}$ does not, by itself, imply local Lipschitz continuity. Likewise, phase retrievability of $\ff$ alone does not guarantee (locally) stable phase retrieval.
	 
	 A large class of applications covered by our framework arises from so-called admissible coherent state systems, as we now explain. Further details, background, and additional applications can be found in \cite{fuehr2005abstract}.
	 \begin{example}[Coherent state systems]\label{ex: coherent state systems vs RKHS}
	 	Let $\Hfrak$ be a Hilbert space, and let $\eta:=(\eta_x)_{x \in X}$ be a family of vectors in $\Hfrak$ indexed by $x \in X$. Here, $X$ is a $\sigma$-compact locally compact Hausdorff space equipped with a Radon-measure $\mu$. Assume that $(\eta_x)_{x \in X}$ is a continuous Parseval frame for $\Hfrak$. That is, for every $f\in \Hfrak$ the coefficient map 
	 	\[X \to \C, \quad x \mapsto \langle f,\eta_x \rangle\] 
	 	is measurable and satisfies
	 	\[
	 	\|f\|^2 = \int_X |\langle f, \eta_x \rangle|^2 \dmu(x).
	 	\] 
	 	In other words, the so-called \emph{voice transform} $V_\eta: \Hfrak \to L^2(X)$ is a well-defined isometry, where 
	 	\[V_\eta f(x) := \langle f, \eta_x \rangle.\] 
	 	The image space 
	 	$\Hcal_\eta:=V_\eta(\Hfrak)\subset L^2(X)$ then turns out to be an RKHS with associated orthogonal projection $V_\eta V_\eta^\ast$ and reproducing kernel $k_\eta$, for $(x,y)\in X\times X$ given by 
	 	\[k_\eta(x,y) = \langle \eta_y,\eta_x \rangle.\] 
	 	
	 	We call $\eta$ an \emph{admissible coherent state system}. Such systems often arise as orbits of a vector under strongly continuous unitary representations of locally compact groups. Prominent examples include the \emph{windowed Fourier transform} and continuous as well as (semi-)discrete generalized \emph{wavelet transforms}.
	 \end{example}
	 
	 The literature on applications of phase retrieval is extensive. We refer to the survey \cite{grohs2020PR} for a thorough introduction and a comprehensive overview of the field up to its time of writing. Since we devote special attention to (semi)-discrete generalized wavelet transforms (see Section~\ref{sec: Cheeger const semi-discrete}), we include one representative application to motivate the subsequent discussion.
	 
	 \begin{example}[Mallat's scattering transform]
	 	Scattering networks---introduced by Mallat in the context of image and audio processing---form a class of structured convolutional neural networks with fixed, predefined filters; see \cite{mallat2012group,bruna2013invariant}. Their purpose is to capture and preserve key structural features of the input signal while suppressing non-informative variability. The basic building block is a cascade of convolutions with wavelet (high-pass) filters interleaved with a pointwise modulus nonlinearity.  
	 	
	 	Let \((\psi_\lambda)_{\lambda\in\Lambda}\subset L^{1}(\R^{d})\cap L^{2}(\R^{d})\) be a countable family of wavelet filters and let \(\chi\in L^{1}(\R^{d})\cap L^{2}(\R^{d})\) be a low-pass filter. For \(\lambda\in\Lambda\) define \(U[\lambda]:L^{2}(\R^{d})\to L^{2}(\R^{d})\) by
	 	\[
	 	U[\lambda]f := |f*\psi_\lambda|.
	 	\]
	 	For a finite path \(p=(\lambda_{1},\dots,\lambda_{n})\in\Lambda^{n}\) set
	 	\[
	 	U[p]f := \bigl|\cdots\,|\,|f*\psi_{\lambda_{1}}|*\psi_{\lambda_{2}}|\cdots *\psi_{\lambda_{n}}\bigr|.
	 	\]
	 	Under mild assumptions on the filter bank, the associated \emph{windowed scattering transform}
	 	\(\Scal:L^{2}(\R^{d})\to \ell^{2}(L^{2}(\R^{d}))\) is given by
	 	\[
	 	\Scal f := \{f*\chi\}\ \cup\ \{U[p]f*\chi : p\in \textstyle{\bigcup}_{n\geq 1}\Lambda^{n}\}.
	 	\]
	 	One also considers \emph{non-windowed} variants of the scattering transform in which one retains (some of) the unaveraged scattering coefficients \(U[p]f\) (or replaces the final averaging by a different pooling); see, e.g., \cite{hirn2017wavelet}. From a mathematical perspective, questions of invertibility are of independent interest, as they shed light on the mapping properties of scattering-type operators~\cite{mallat2016understanding,waldspurger2017wavelet}. Consequently, invertibility and stability questions for scattering-type representations are closely tied to phase retrieval problems for measurement families of the form \((|f*\psi_\lambda|)_{\lambda\in\Lambda}\). 
	 	
	 	Numerical inversion of scattering transforms has been reported, for instance, in audio processing \cite{anden2014deep} and on natural image data sets \cite{angles2018generative}.
	 	For Cauchy wavelets, injectivity and weak stability of the associated phase retrieval problem can be proved using tools from complex analysis \cite{mallat2015phase}. In the real-valued case, sign retrieval is possible under suitable bandlimitation assumptions on the filters, together with (weak) continuity of the reconstruction map \cite{alaifari2017reconstructing}. 
	 \end{example}
	 
	 The measurements contain only phaseless information about the unknown signal \(F\). It is therefore desirable to have a local stability certificate that depends only on \(|F|\). Informally, one would like to infer from the modulus \(|F|\) how well \(F\) can be recovered from phaseless data. Ideally, any comparison of the stability constant $C(F)$ 
	 should be expressible in terms of interpretable properties of \(|F|\). 
	 A central notion in this direction is that of a Cheeger-type constant, introduced into the phase retrieval literature in~\cite{grohs2019spectral,grohs2021stable} and developed further in~\cite{grohs2021L2stability,cheng2021stable,rathmair2024stable,bartusel2024role}, as well as in the recent preprint~\cite{alaifari2025cheeger}. 
	 
	 Cheeger constants provide quantitative measures of connectedness (or, more precisely, of the presence of bottlenecks) for geometric and combinatorial objects, such as Riemannian manifolds \cite{cheeger2015lower} and graphs \cite{chung1996laplacians}. In phase retrieval it is a folklore that instabilities are mainly caused by a lack of connectivity in the measurement data; see, e.g., \cite{alaifari2019stable,alaifari2021gabor}. 
	 
	 Our results differ from the existing literature in several aspects:
	 \begin{itemize}
	 	\item In \cite{cheng2021stable} the analysis starts from a setting in which  phase retrieval is already (locally) stable with respect to \emph{local measurements}. Global stability is then obtained by piecing together these local estimates under an additional global connectivity condition. In contrast, we do not assume local stability a priori. Instead, working in a general RKHS framework, we characterize the local stability constant \(C(F)\) in terms of the reciprocal of a suitably defined Cheeger constant.
	 	
	 	\item In \cite{grohs2019spectral,grohs2021stable,grohs2021L2stability,rathmair2024stable} it is shown that stability for STFT phase retrieval is controlled by the reciprocal of a Cheeger constant measuring the connectivity of the phaseless data (i.e.\ of \(|F|\) in our notation).  Furthermore, there exist two dense families of signals, one of which has associated Cheeger constant zero and the other strictly positive~\cite{alaifari2025cheeger}. 
	 	
	 	The scope of these results is, however, restricted to phase retrieval from STFT measurements (and to specific window classes). Our framework is substantially more general: it includes STFT phase retrieval as a special case, but also covers further phase retrieval scenarios; see Example~\ref{ex: coherent state systems vs RKHS}. Moreover, our formulation is regularity-preserving, i.e., our notion of Cheeger constant is defined at the level of the measurements without introducing additional smoothing or differentiation, so the analysis stays in the same regularity class as the signal generating the data.
	 	
	 	\item The works \cite{grohs2019spectral,grohs2021stable,grohs2021L2stability,cheng2021stable,rathmair2024stable} provide \emph{upper} bounds on the local stability constant in terms of the reciprocal (setting-specific) Cheeger constant, whereas \cite{bartusel2024role} derives a \emph{lower} bound but not an upper bound. Upper bounds of this type yield uniqueness for signals with positive Cheeger constant as a byproduct. By contrast, a lower bound shows that a Cheeger constant bounded away from zero is not only sufficient but also necessary for (local) stability. To the best of our knowledge, our results are the first to yield both upper and lower bounds in the real case, and lower bounds in the complex setting, within a unified RKHS framework. Moreover, our approach is dimension-agnostic: it treats finite- and infinite-dimensional settings on the same footing, and it applies to both (semi-)discrete and continuous measurement domains.
	 	
	 	\item Finally, our proofs rely primarily on elementary Hilbert space methods. In particular, they avoid some of the more specialized tools (e.g.\ complex-analytic arguments or Poincar\'e-type inequalities) that play a central role in~\cite{grohs2019spectral,grohs2021stable,grohs2021L2stability,rathmair2024stable,alaifari2024unique,alaifari2025cheeger}.
	 \end{itemize}
	 Our main contributions are:
	 \begin{itemize}
	 	\item \textbf{Injectivity for generalized wavelet phase retrieval:} Before turning to the stability theory, which is the main focus of this paper, we first address injectivity for generalized wavelet phase retrieval. In Section~\ref{sec: PR conv} we provide several characterizations of (local) phase retrievability and derive a simple sufficient criterion for wavelet sign retrieval in arbitrary dimension for transforms associated with irreducibly admissible matrix groups, thereby extending parts of the injectivity discussion in \cite{alaifari2017reconstructing,alaifari2024unique}.
	 	
	 	\item \textbf{Cheeger-based stability in RKHSs:} We develop a general RKHS stability theory for phase retrieval in which the local stability constant \(C(F)\) is characterized in terms of the reciprocal of a kernel Cheeger constant. In particular, we obtain both upper and lower stability bounds in the real case, and lower bounds in the complex case (Theorems~\ref{thm: RKHS PR lower Cheeger bound} and~\ref{thm: RKHS SR upper bound}).
	 	
	 	\item \textbf{Kernel versus graph Cheeger constants:} In the setting of generalized wavelet phase retrieval from (semi-)discrete frames, we prove an estimate relating our kernel Cheeger constant to the classical graph Cheeger constant of a suitably defined data-dependent weighted graph (Theorem~\ref{thm: kernel vs. graph Cheeger constant semi-discrete}).
	 \end{itemize}

\section{Phase retrieval from convolutions} \label{sec: PR conv}

Let $A$ be a second-countable locally compact abelian Hausdorff (SLCA) group. 
For a family of filters $\Psi = (\psi_\lambda)_{\lambda \in \Lambda} \subset L^2(A)$, the associated \emph{generalized wavelet transform} $W_\Psi f$ of $f\in L^2(A)$ is given by 
\[
W_\Psi f:A \times \Lambda \to \C, \quad  (x,\lambda) \mapsto (f \ast \psi_\lambda^*)(x).
\]
We assume $\Lambda$ to be a $\sigma$-compact locally compact Hausdorff space, 
equipped with a Radon measure $\nu$. Moreover, we impose that $\Psi$ is weakly Borel measurable; that is, for all $f \in L^2(A)$, the map 
\[
\Lambda \to \C, \quad \lambda \mapsto \langle f, \psi_\lambda \rangle 
\] 
is Borel measurable. We denote by $\mathfrak{B}_\nu(\Lambda)$ the set of all Borel measurable functions, where we identify two functions if they agree $\nu$-almost everywhere. 

The phase retrieval problem for $W_\Psi$ becomes: Recover $f$ from  $|W_\Psi f|$ (up to a unimodular scalar multiple).

For later use, we note the following consequence of the Borel measurability of the family $\Psi$.
\begin{lemma} \label{lem:norm_Borel}
	If $\Psi = (\psi_\lambda)_{\lambda \in \Lambda}$ is weakly Borel measurable, then for every $f \in L^2(A)$ the map 
	\[
	\Lambda \to [0,\infty], \quad \lambda \mapsto \| f \ast \psi_\lambda^* \|_2
	\]
	is Borel measurable.
\end{lemma}
\begin{proof}
	If $\Psi$ is weakly Borel measurable, then the mapping
	\[
	A \times \Lambda \to \C, \quad (x,\lambda) \mapsto \langle f, T_x \psi_\lambda \rangle 
	\] 
	is a Carath\'eodory function by the strong continuity of the regular representation of $A$ on $L^2(A)$. Note that, being an SLCA group, $A$ is (completely) metrizable and separable.
	Hence, the map is Borel measurable with respect to the product sigma algebra by Carath\'eodory's theorem. Applying Tonelli's theorem then shows that the function
	\[
	\lambda \mapsto \| f \ast \psi_\lambda^* \|_2^2 = \int_A |\langle f, T_x \psi_\lambda \rangle|^2 \dx \in [0,\infty]
	\] is also Borel measurable. 
\end{proof}

We consider the following list of conditions on the filters: 
\begin{description}
	\item[Calder\'on condition \textup{(C)}] $\Psi$ satisfies the generalized Calder\'on condition
	\[
	\int_\Lambda |\widehat{\psi}_\lambda(\xi)|^2 \dnu(\lambda) = 1 \quad \text{a.e. } \xi \in \widehat{A}.
	\]
	\item[Spectral injectivity condition \textup{(SI)}] For every $\lambda\in \Lambda$, one has $\widehat{\psi}_\lambda \in L^\infty(\widehat{A})$, and the map 
	\[
	M_\Psi: L^1(\widehat{A}) \to \mathfrak{B}_\nu(\Lambda),\quad (M_\Psi F)(\lambda) = \int_{\widehat{A}} F(\xi) \, |\widehat{\psi}_\lambda(\xi)|^2 \dxi, \] is injective. 
\end{description}

The following result clarifies the role of these conditions in generalized wavelet phase retrieval. In particular, it shows that the spectral injectivity condition \textup{(SI)} essentially derives a whole family of Fourier phase retrieval problems.

\begin{theorem}\label{thm: characterization of properties C and P}
	\begin{enumerate}
		\item[(a)] The family $\Psi$ satisfies the Calder\'on condition \textup{(C)} if and only if the operator $W_\Psi : L^2(A) \to L^2(A \times \Lambda)$ is an isometry, and this is equivalent to the weak-sense inversion formula 
		\[
		f = \int_{\Lambda} f \ast \psi_\lambda^* \ast \psi_\lambda \dnu(\lambda), 
		\] holding for all $f \in L^2(A)$. 
		\item[(b)] Assume that $\Psi$ satisfies the spectral injectivity condition \textup{(SI)}, and let $f, g \in L^2(A)$. Then $|W_\Psi f| = |W_\Psi g|$ holds if and only if the following two properties are satisfied:  
		\begin{enumerate}
			\item[(i)] $|\widehat{f}| = |\widehat{g}|$, 
			\item[(ii)] $\forall \lambda \in \Lambda\,:\,|f \ast \psi_\lambda^\ast| = |g \ast \psi_\lambda^\ast|$ and $\big| \widehat{f \ast \psi_\lambda^\ast} \big| = \big| \widehat{g \ast \psi_\lambda^\ast} \big|$. 
		\end{enumerate}
	\end{enumerate}
\end{theorem}

\begin{proof}
	Part (a) is a folklore consequence of the Plancherel and convolution theorems.
	
	For part (b), assume that $|W_\Psi f| = |W_\Psi g|$. Then the convolution and the Plancherel theorem yield, for all $\lambda \in \Lambda$, 
	\begin{eqnarray*}
		\int_{\widehat{A}} |\widehat{f}(\xi)|^2 \, |\widehat{\psi}_\lambda(\xi)|^2 \dxi  & = &
		\int_{A} |(f \ast \psi_\lambda^\ast)(x)|^2 \dx \\
		& = &  \int_{A} |(g \ast \psi_\lambda^\ast)(x)|^2 \dx 
		 =   \int_{\widehat{A}} |\widehat{g}(\xi)|^2 \, |\widehat{\psi}_\lambda(\xi)|^2 \dxi.
	\end{eqnarray*} Hence $M_\Psi (|\widehat{f}|^2) = M_\Psi (|\widehat{g}|^2)$, and thus $|\widehat{f}|^2 = |\widehat{g}|^2$ by the spectral injectivity condition \textup{(SI)}. This shows (i). Now (ii) follows from (i) and the convolution theorem. 
\end{proof}

\begin{remark}\label{rm: conditions (C) and (SI)}
	The Calder\'on condition \textup{(C)} and methods for guaranteeing it are well-understood in the contexts of time-frequency analysis and (generalized) wavelet transforms. For instance, in the case of time frequency analysis, we have
	\[
	\Lambda = \widehat{A},\quad \psi_\lambda = M_\lambda \psi,\quad \nu = \mbox{Haar measure on } \widehat{A}
	\] and any $\psi \in L^2(A)$ with $\| \psi \|_2 = 1$ creates a system satisfying the Calder\'on condition \textup{(C)}. 
	
	In the wavelet setting we take \(A=\R^{d}\) and \(\Lambda = H\le \GL_d(\R)\), where \(H\) is a suitably chosen matrix group. The filters are of the form
	\[
	\psi_{\lambda}=|\det(\lambda)|^{-1}\,\psi(\lambda^{-1}\,\cdot),\quad \lambda\in H,
	\]
	and we equip \(H\) with its left Haar measure \(\nu\). Under mild assumptions on \(\psi\) (depending on the choice of \(H\)), the Calder\'on condition \textup{(C)} is satisfied.

	The spectral injectivity condition \textup{(SI)} is motivated by the phase retrieval problem and, to the best of our knowledge, is new in this context. It is not immediately clear under which circumstances this condition holds, although several simple observations can be made in concrete cases.
	
	In the time-frequency setting, note that $M_\Psi F = F \ast |\widehat{\psi}|^2$. It is a standard observation that a convolution operator is injective if and only if the Fourier transform of the convolution kernel vanishes almost nowhere. Specifically, if $V_\psi f$ denotes the windowed Fourier transform of $f$ with respect to the window $\psi$, then Theorem~\ref{thm: characterization of properties C and P} yields that $|V_\psi f| = |V_\psi g|$ implies $|\widehat{f}| = |\widehat{g}|$ whenever $\psi \ast \psi^*$ vanishes almost nowhere.
	We expect that this observation can be deduced by more direct means, e.g., by proper integration over $|V_\psi f|^2$. 
	In any case, this condition is satisfied if $A = \mathbb{R}^d$ and $\psi$ is bandlimited, i.e., $\widehat{\psi}$ is compactly supported. 
	
	In the wavelet setting, first note that $|\widehat{\psi}_\lambda|^2$ is symmetric if $\psi$ is real. In this case, the spectral injectivity condition \textup{(SI)} is not satisfied. However, for the one-dimensional setting with $H=\R^\ast$, the map $M_\Psi$ is related to a convolution operator over $H$, and we can say the following: Assume that $\psi$ has compact support $\supp(\widehat{\psi}) \subset \R_{> 0}$ (away from zero). Then $\Psi$ satisfies the spectral injectivity condition \textup{(SI)}. We provide the proof of this claim in the appendix, see Proposition~\ref{prop: sufficient condition for (SI)}. 
\end{remark}

We next want to introduce definitions that are connected with strategies for phase retrieval, or at least for better understanding the problem. These notions rely on the idea of separating the phase retrieval problem into a ``local'' problem that addresses recovering \(f*\psi_\lambda^\ast\) from \(|W_\Psi f|\) for each \(\lambda\) separately, and a ``global'' problem that asks how many ways there are to synthesize the local solutions into a global one. This local--global separation has been formulated in various places in the literature; see, e.g.,~\cite{alaifari2017reconstructing,cheng2021stable,bartusel2024role}. It is important to note that ``local'' here refers to the variable $\lambda \in \Lambda$, which indexes a partition of unity in the frequency domain. 

\begin{definition}
	For any $f \in L^2(A)$, we define $\Lambda_f := \{ \lambda \in \Lambda : \| f \ast \psi_\lambda \|_2 \not= 0 \}$. 
	Moreover, we say that $f$ is \emph{locally phase retrievable} if for every $\lambda\in \Lambda$ and every $g\in L^2(A)$ it holds that
	\[
	|f*\psi_\lambda^\ast|=|g*\psi_\lambda^\ast| \quad \Rightarrow \quad g \ast \psi_\lambda^* \in \T \cdot (f \ast \psi_\lambda^*).  
	\]
\end{definition}

An alternative way of formulating local phase retrievability is to say that, up to phase factors, all solutions to the $\lambda$-local phase retrieval problem are of the form $f + g$, where $g \ast \psi_{\lambda} = 0$. This observation stresses that the separation of the phase retrieval problem into a local and a global one only makes sense if the involved filters are bandlimited. The following definition introduces language that makes this observation more precise. 

\begin{definition}\label{def: approx_f and sim_f}
	Let $f \in L^2(A)$. We introduce an equivalence relation on $\Lambda$, defined as the transitive hull $\sim_f$ of the symmetric, reflexive relation 
	\[
	\lambda \approx_f \lambda' :\Leftrightarrow \| f \ast \psi_\lambda^* \ast \psi_{\lambda'}^* \|_2 > 0. 
	\]
	The equivalence class modulo $\sim_f$ of $\lambda \in \Lambda$ is denoted by $[\lambda]_{\sim_f}$.
\end{definition}
The convolution theorem and  $\widehat{g*} = \overline{\widehat{g}}$ imply $\| f \ast g \|_2  = \| f \ast g^*\|_2 $, for all $f, g\in L^2(A)$. This allows to replace $\psi_\lambda^*$ by $\psi_\lambda$ and $\psi_{\lambda'}^*$ by $\psi_{\lambda'}$, wherever this may be helpful.

The following lemma formulates a number of fundamental observations in connection with $\sim_f$. Parts (b) and (c) point out a general construction of counterexamples to unique phase retrieval. These illustrate that a necessary condition for phase retrieval is the existence of a ``large'' equivalence class $[\lambda]_{\sim_f} \subset \Lambda_f$. Clearly this property is related to the support properties of the filters in Fourier domain. Conversely, part~(d) provides a route towards uniqueness via \emph{phase propagation}: phase information can be propagated within each $\sim_f$-equivalence class, and larger classes correspond to fewer remaining ambiguities for the given signal.
\begin{lemma}\label{lem:local_global}
	Assume that $\Psi$ satisfies the Calder\'on condition \textup{(C)}, and let $f \in L^2(A)$. 
	\begin{enumerate}
		\item[(a)] $\Lambda_f$ is a Borel-measurable union of $\sim_f$-equivalence classes. 
		\item[(b)] Let $U \subset \Lambda$ denote any Borel measurable union of $\sim_f$-equivalence classes. Define 
		\[
		f_U := \int_U f \ast \psi_\lambda^* \ast \psi_\lambda \dnu(\lambda). 
		\]
		Then $f_U \in L^2(A)$, with 
		\[
		f_U \ast \psi_{\lambda^\prime}^* = \begin{cases}
			f \ast \psi_{\lambda^\prime}^* & \lambda^\prime \in U \\ 0 & \lambda^\prime \not \in U
		\end{cases}.
		\]
		\item[(c)] Assume that $\Lambda = \bigcup_{j \in J} U_j$, with pairwise disjoint Borel sets $U_j$ that are unions of $\sim_f$-equivalence classes. Let $(\sigma_j)_{j \in J} \in \T^J$ be arbitrary. Then 
		\[
		g = \sum_{j \in J} \sigma_j f_{U_j} 
		\] converges unconditionally in $L^2(A)$, with 
		\begin{equation} \label{eqn:changesign}
			\forall j \in J \,\forall \lambda \in U_j\,:\,  W_\Psi g(\cdot,\lambda) = \sigma_j W_\Psi f (\cdot,\lambda).
		\end{equation}
		\item[(d)] Assume that $f$ is locally phase retrievable. Let $g \in L^2(A)$ with $|W_\Psi g| = |W_\Psi f|$. Then there exists a family $(\sigma_{\lambda})_{\lambda \in \Lambda} \in \T^{\Lambda}$ such that 
		\[
		\forall \lambda \in \Lambda \,:\, g \ast \psi_{\lambda}^* = \sigma_{\lambda} (f \ast \psi_{\lambda}^*).
		\] 
		Furthermore, the map $\lambda \mapsto \sigma_\lambda$ is constant on $\sim_f$-equivalence classes. 
	\end{enumerate}
\end{lemma}

\begin{proof}
	Proof of (a): Measurability of $\Lambda_f$ follows from Lemma~\ref{lem:norm_Borel}. For $\lambda \in \Lambda \setminus \Lambda_f$, one finds $[\lambda]_{\sim_f} = \{ \lambda \}$, hence $\Lambda \setminus \Lambda_f$ is a union of equivalence classes. The same then holds for $\Lambda_f$. 
	
	Proof of (b): Given $U$, let 
	\[
	F_U: A \times \Lambda \to \mathbb{C},\quad F_U(x,\lambda) =  (f \ast \psi_\lambda^*)(x) \cdot \mathds{1}_{U}(\lambda).
	\]
	Then $F_U \in L^2(A \times \Lambda)$ due to the Calder\'on condition \textup{(C)}, hence $W_\Psi^\ast(F_U)\in L^2(A)$. Now,
	\begin{align*}
		W_\Psi^*(F_U) = \int_{\Lambda} F_U (\cdot ,\lambda) \ast \psi_\lambda \dnu(\lambda) 
		= \int_U W_\Psi f(\cdot,\lambda) \ast \psi_\lambda \dnu(\lambda)
		= f_U,
	\end{align*} 
	where the integral converges in the weak sense.
	
	Picking any $\lambda' \in U$, we have 
	\begin{align*}
		f_U \ast \psi_{\lambda'}^* = \int_U W_\Psi f(\cdot,\lambda) \ast \psi_\lambda \ast \psi_{\lambda'}^* \dnu(\lambda)
		=   \int_U f \ast \psi_\lambda^* \ast \psi_\lambda \ast \psi_{\lambda'}^* \dnu(\lambda).
	\end{align*}
	Since $[\lambda']_{\sim_f} \subset U$, it follows that, for all $\lambda \in \Lambda \setminus U$, 
	\[
	\| f \ast \psi_\lambda^* \ast \psi_{\lambda'} \|_2 = 0.
	\] 
	Thus, using the weak-sense inversion formula for $f$, we find that
	\begin{align*}
		\int_U f \ast \psi_\lambda^* \ast \psi_\lambda \ast \psi_{\lambda'}^* \dnu(\lambda) 
		= \int_\Lambda f \ast \psi_\lambda^* \ast \psi_\lambda \ast \psi_{\lambda'}^* \dnu(\lambda)
		= f \ast \psi_{\lambda'}^*.
	\end{align*} 
	If instead $\lambda' \in \Lambda \setminus U$ and $\lambda \in U$, the assumption that $U$ is a union of $\sim_f$-equivalence classes yields 
	\[ f \ast \psi_{\lambda} \ast \psi_{\lambda'}^* = 0 \]
	and therefore 
	\begin{align*}
		f_U \ast \psi_{\lambda'}^* = \int_U f \ast \psi_\lambda^*  \ast \psi_\lambda \ast \psi_{\lambda'}^* \dnu(\lambda)  = 0. 
	\end{align*}
	This shows part (b). 
	
	For part (c), part (b) implies that $(\sigma_j f_{U_j})_{j \in J}$ is a family in $L^2(A)$. Pairing part (b) also with the isometry property of $W_\Psi$ yields, for all $j \not= k$,  
	\[
	\langle f_{U_j}, f_{U_k} \rangle = \langle W_\Psi f_{U_j}, W_\Psi f_{U_k} \rangle = 0. 
	\] 
	Since the $U_j$ are disjoint and cover all of $\Lambda_f$, we obtain 
	\[
	\sum_{j \in J} \|f_{U_j} \|_2^2 = \sum_{j \in J} \| W_\Psi f_{U_j} \|_2^2 = \| W_\Psi f \|_2^2 = \| f \|_2^2.
	\]
	Hence
	\[
	g = \sum_{j \in J} \sigma_j f_{U_j}
	\] is an unconditionally converging orthogonal sum, for every choice $(\sigma_j)_{j \in J} \in \T^J$, and (\ref{eqn:changesign}) follows from (b) and disjointness of the $U_j$.

	Proof of (d): 
	The existence of the family $(\sigma_\lambda)_{\lambda \in \Lambda}$ with the desired properties follows from the definition of local phase retrievability. 
	
	For $\lambda \not= \lambda'$, associativity and commutativity of convolution allow to compute
	\begin{align*}
		f \ast \psi_\lambda^\ast \ast \psi_{\lambda'}^\ast = \sigma_\lambda^{-1} (g \ast \psi_\lambda^\ast \ast \psi_{\lambda'}^\ast) 
		= \sigma_{\lambda'} \sigma_{\lambda}^{-1} (f \ast \psi_\lambda^\ast \ast \psi_{\lambda'}^\ast),
	\end{align*}
	which for $\lambda \approx_f \lambda'$ entails $\sigma_\lambda = \sigma_{\lambda'}$. Since $\sim_f$ is the transitive hull of $\approx_f$, we get $\sigma_\lambda = \sigma_{\lambda'}$ for all $\lambda' \in [\lambda]_{\sim_f}$. 
\end{proof}

We can combine the observations collected in Lemma~\ref{lem:local_global} to obtain the following characterization of phase retrievability.

\begin{corollary}\label{cor: PR vs local PR}
	Assume that $\Psi$ satisfies the Calder\'on condition \textup{(C)}, and let $f \in L^2(A)$. 
	Then $f$ is phase retrievable if and only if the following two properties hold:
	\begin{enumerate}
		\item[(a)] $f$ is locally phase retrievable; 
		\item[(b)] If $\Lambda_f = U \cup V$, with $U, V$ disjoint Borel unions of $\sim_f$-equivalence classes, then $\nu(U) = 0$ or $\nu(V) = 0$. 
	\end{enumerate}
\end{corollary}

\begin{proof}
	That phase retrievability of $f$ implies condition (a) is obvious. Furthermore, any counterexample $U,V \subset \Lambda_f$ to condition (b) yields a counterexample to phase retrievability, in the form of 
	\[
	g = f_U - f_V.
	\] Here we note that 
	\[ \| f_U \|_2^2 = \int_U \| f \ast \psi_{\lambda}^\ast \|_2^2 \,\dnu(\lambda) >0 \]
	by assumption on $U$ and definition of $\Lambda_f \supset U$, and $\| f_V \|_2^2>0$ follows analogously.  
	
	Conversely, assume that (a) and (b) hold, and let $g \in L^2(A)$ be such that $|W_\Psi f| = |W_\Psi g|$ holds. By Assumption (a), this implies that, for all $\lambda \in \Lambda_f$,
	\[
	W_\Psi g(\cdot,\lambda) = \sigma_\lambda W_\Psi f(\cdot,\lambda),
	\]
	with a Borel measurable map $\Lambda_f \ni \lambda \mapsto \sigma_\lambda \in \T$.
	
	It suffices to show that $\sigma_\lambda$ is $\nu$-a.e. constant on $\Lambda_f$. Suppose not. Then there exists a Borel subset $C\subset \T$ such that both sets
	\[
	U = \{ \lambda \in \Lambda_f : \sigma_\lambda \in C \},\quad V = \Lambda_f \setminus U
	\] 
	have positive $\nu$-measure. By Lemma~\ref{lem:local_global}, $U$ and $V$ are unions of $\sim_f$-equivalence classes, contradicting assumption (b). 
	
	Therefore $\lambda \mapsto \sigma_\lambda$ is $\nu$-a.e. constant on $\Lambda_f$, say with value $\sigma\in \T$. Now  $g = \sigma f$ follows from the inversion formula. 
\end{proof}

The following lemma lays out the ground for the remaining analysis in this section.

\begin{lemma}\label{lem:bl_sr}
	Assume that the field of scalars is $\F=\R$. If $\Psi$ is a family of bandlimited filters satisfying the Calder\'on condition \textup{(C)}, then every $f \in L^2(\R^d)$ is locally sign retrievable in $L^2(\R^d)$.
\end{lemma}
\begin{proof}
	This follows from the fact that for $g_1,g_2 \in L^2(\mathbb{R}^d)$ bandlimited and real-valued, satisfying $|g_1| = |g_2|$, one has $g_1 \in \{\pm g_2\}$, see e.g. \cite{akutowicz1957determination,alaifari2017reconstructing}.
\end{proof}

Hence, the sign retrieval problem fundamentally depends on the properties of the equivalence relation $\sim_f$. It is important to realize that this not only depends on $\widehat{f}$, but also on the supports of the $\widehat{\psi}_\lambda$ and their overlaps. We discuss a few general example classes to illustrate this phenomenon.

\begin{example}
	We consider real-valued \emph{Shannon wavelets} on $L^2(\mathbb{R})$, given by $\Psi = (\psi_j)_{j \in \mathbb{Z}}$ with 
	\[
	\widehat{\psi}_j = \mathds{1}_{[-2^{j},-2^{j-1}]} + \mathds{1}_{[2^{j-1},2^{j}]}.
	\] Then $\Psi$ satisfies the Calder\'on condition \textup{(C)} with counting measure $\nu$. Moreover, $f \in L^2(\R)$ is locally sign retrievable in $L^2(\R^d)$ by Lemma~\ref{lem:bl_sr}. Furthermore, the fact that $\psi_j \ast \psi_k = 0$ whenever $j \not= k$ implies that all $\sim_f$-equivalence classes are trivial, for all $f \in L^2(\R)$. 
	Accordingly, all solutions $g \in L^2(\R)$ to the sign retrieval problem $|W_\Psi f| = |W_\Psi g|$ are obtained by 
	picking an arbitrary subset $U \subset \mathbb{Z}$, and writing
	\[
	g_U = \sum_{j \in U} f \ast \psi_j^* \ast \psi_j - \sum_{k \in \mathbb{Z} \setminus U} f \ast \psi_k^* \ast \psi_k.
	\]
	
	Hence the sign retrieval problem is not uniquely solvable in $L^2(\R)$ even for nice functions. To see this point, consider the space of all compactly supported (real-valued) functions
	$L^2_c(\mathbb{R})$. For $0\neq f \in L^2_c(\mathbb{R})$, we have $\| f \ast\psi_j^* \|_2>0$ for all $j \in \mathbb{Z}$ since the zeros of $\widehat{f}$ are isolated. As a consequence, the map $\mathbb{Z} \supset U \mapsto g_U$ is injective, i.e., there exist uncountably many fundamentally different solutions to the sign retrieval problem for $f$. 
	
	It is worthwhile noting that the sign retrieval problem for $f$ becomes uniquely solvable if one imposes that the solution belongs to $L^2_c(\mathbb{R})$ as well. The reason for this is that the local phase retrieval property of the system allows to reconstruct the restriction of $\widehat{f}$ on $\pm [2^{j-1},2^j]$ up to a local sign $\sigma_j$, and the fact that $\widehat{f}$ is analytic then fixes this sign globally, for all $j$. Note however that this conclusion was not provided by Lemma~\ref{lem:local_global} (d).
	
	A slight modification of this construction is obtained by introducing 
	\[
	\widehat{\phi}_j = \mathds{1}_{[-2^{j}(1+\epsilon),-2^{j-1}(1-\epsilon)]} + \mathds{1}_{[2^{j-1}(1-\epsilon),2^{j}(1+\epsilon)}]
	\]
	for $0< \epsilon<1$, and then letting
	\[
	\widehat{\psi}_j(\xi) = \frac{\widehat{\phi}_j(\xi)}{\sum_{k \in \mathbb{Z}} |\widehat{\phi}_k(\xi)|^2}.
	\] By construction, the system $\Psi$ satisfies the Calder\'on condition \textup{(C)}. Now for any $0\neq f \in L^2_c(\mathbb{R})$, we have $\| f \ast \psi_j \ast \psi_{j+1}\|_2 >0$ for all $j \in \mathbb{Z}$, and thus $\Lambda_f = \mathbb{Z}$ is a single $\sim_f$-equivalence class. Accordingly, $f$ is (uniquely) sign retrievable. 
\end{example}

The next class of constructions provides systems $\Psi$ with a high degree of redundancy. 
\begin{definition}
	Let $H \leq \GL_d(\R)$ denote an \emph{irreducibly admissible matrix group}, i.e., there exists a unique open set $\mathcal{O}\subset \R^d$ (which necessarily is of full measure) such that $\mathcal{O} = H^T \xi_{\Ocal}$ for a suitable $\xi_\Ocal \in \mathbb{R}^d$, and in addition, the dual stabilizer 
	\[
	H_{\xi_\Ocal} = \{ h \in H : h^T \xi_\Ocal = \xi_\Ocal \} \subset H
	\] is compact. $\mathcal{O}$ is called \emph{open dual orbit} of $H$. 
\end{definition}

\begin{remark}\label{rm: wavelet filters from admissible matrix group}
	Whenever $H$ is an irreducibly admissible matrix group with open dual orbit $\mathcal{O}$, one may pick any $0\neq \psi \in L^2(\mathbb{R}^d)$ such that $\widehat{\psi}$ is compactly supported inside $\mathcal{O}$, and define
	\[
	\Psi = (\psi_h)_{h \in H},\quad \psi_h = |{\rm det}(h)|^{-1} \,\psi(h^{-1}\,\cdot). 
	\]
	By renormalizing $\psi$, if necessary, one can then achieve
	\[
	\int_H |\widehat{\psi}(h^T \xi)|^2 \dmu_H (h) = 1 \quad \text{a.e. } \xi \in \R^d,
	\] and obtains that $\Psi$ satisfies the Calder\'on condition \textup{(C)}, with measure $\nu = \mu_H$, the left Haar measure on $H$. For details and additional background, we refer to \cite{fuehr2010generalized,fuehr2005abstract}. Note that our normalization of the wavelets, so that $h \mapsto \| \psi_h \|_1$ is constant, leads to a slightly different measure in the inversion formula compared with \cite{fuehr2005abstract}.
	
	One can show that $\mathcal{O} = -\mathcal{O}$, and this allows to choose a \emph{real-valued} bandlimited wavelet $\psi$, e.g., by picking any bandlimited wavelet $\psi_c$ and letting $\psi = {\rm Re}(\psi_c)$. We are therefore in a position to conclude local sign retrievability via Lemma~\ref{lem:bl_sr}. 
\end{remark}

We recall that, for $g \in L^p(\mathbb{R}^d)$, its (essential) support 
\begin{align*}
	\supp(g) = \R^d \setminus \textstyle{\bigcup} \bigl\{U\subset \R^d\,\big|\, U \text{ open},\, g= 0 \text{ a.e. on } U\bigr\}
\end{align*}
does not depend on the choice of representative, i.e., $\supp(g)=\supp(h)$ holds whenever $h=g$ a.e. Hence it is well-defined to speak of (topological) connectedness of $\supp(g)$. 

Motivated by \cite[Theorem~1]{alaifari2017reconstructing}, we end this section with a sufficient criterion for sign retrieval from wavelet transforms associated with irreducibly admissible matrix groups. Our criterion does not impose exponential decay of \(f\); cf.\ \cite[Remark~3]{alaifari2017reconstructing}. Moreover, it applies in arbitrary dimension (cf.\ \cite{alaifari2023phase}) and does not require the real-analyticity assumptions of \cite[Theorem~3.1]{alaifari2024unique}.

\begin{theorem}\label{thm: PR bandlimited wavelets, real case}
	Assume that the field of scalars is $\F=\R$. Let $H$ be an irreducibly admissible matrix group with associated open dual orbit $\Ocal$, and 
	bandlimited (real-valued) wavelets $\Psi = (\psi_h)_{h \in H}$ as described in Remark~\ref{rm: wavelet filters from admissible matrix group}. Let $0\not= f \in L^2(\mathbb{R}^d)$. 
	If the boundaries of $F:=\inter(\supp(\widehat{f}))$ and $V:=\inter(\supp(\widehat{\psi}))$ are null-sets, and if $F\cap \Ocal$ is connected, then $f$ is uniquely sign retrievable. 
\end{theorem}
\begin{proof}
	Observe that in this setting $\Lambda=H$, so formally $\Lambda_f=H_f$.
	
	The proof is based on the idea that if $F\cap \Ocal$ is connected, then $H_f$ consists of a single $\sim_f$-equivalence class. Applying Corollary~\ref{cor: PR vs local PR} and Lemma~\ref{lem:bl_sr} then concludes the proof of this theorem.
	
	Let us now show that $H_f$ consists of a single $\sim_f$-equivalence class. To this end, first recall that $\widehat{\psi}$ is (compactly) supported in $\Ocal=H^T\xi_\Ocal$. Hence, $V=H_V^T\xi_\Ocal$ for some non-empty subset $H_V\subset H$, implying that $\Ocal$ can be covered by the sets $V_h:=h^{-T}V=\inter(\supp(\widehat{\psi}_h))$, $h\in H$. By Lemma~\ref{lem:local_global}, $H_f$ is a union of $\sim_f$-equivalence classes $C$. Moreover, employing that the boundaries of $F$ and $V$ are null-sets, one quickly finds that $F\cap V_h$ is nonempty if and only if $h \in H_f$. Together,
	\begin{align*}
		F\cap \Ocal = \bigcup_{h\in H_f} (F\cap V_h) = \biguplus_{c\in C}\bigcup_{h\in c} (F\cap V_h).
	\end{align*}  
	But $F\cap \Ocal$ is connected and, for each equivalence class $c\in C$, the union $\bigcup_{h\in c} (F\cap V_h)$ is a nontrivial open subset thereof; hence, there can only be one such class.
\end{proof}

\begin{remark}
	Suppose that the open dual orbit $\Ocal$ is connected---for instance, this is the case when $H$ is connected. Then $F\cap \Ocal$ is connected if and only if $F$ is connected. Specifically, our theorem then applies to any (real-valued) compactly supported function $0\not= f \in L^2_c(\mathbb{R}^d)$. Indeed, $\widehat{f}$ is easily seen to be extended to a complex analytic function. By the identity theorem for complex-analytic functions in multiple variables, this extension cannot vanish on any nontrivial open subset of $\R^d$ unless it is identical zero. In particular, the zero-set of $\widehat{f}$ must have empty interior, entailing $\supp(\widehat{f})=\R^d=F$. 
\end{remark}

\section{A Kernel Cheeger Constant for Locally Stable Phase Retrieval in RKHSs}\label{sec: cheeger RKHS}
	In this section, we take a step back and consider the phase retrieval problem in a more general framework, assuming that the signals to be recovered belong to a reproducing kernel Hilbert space. We return in Section~\ref{sec: Cheeger const semi-discrete} to the more specific convolutional phase retrieval problem associated with generalized wavelet transforms.
	
	Let $(X,\Sigma,\mu)$ be a $\sigma$-finite 
	measure space. We assume that $\Hcal\subset \F^X$ is an RKHS that embeds isometrically into a (necessarily closed) subspace of $L^2(X):=L^2(X,\Sigma,\mu)$. That is, $\Hcal$ is a Hilbert space of functions $X\to \F$ such that, for each $x\in X$, the point evaluation $\Hcal\ni \hf \mapsto \hf(x)\in \F$ is a well-defined bounded linear functional, and the embedding $\iota: \Hcal \hookrightarrow L^2(X)$ is isometric. By Riesz' representation theorem, there exists $k_x\in \Hcal$ such that $\hf(x)=\ip{\hf,k_x}$ for all $\hf\in \Hcal$. The function $k_x$ is called the \emph{reproducing kernel for the point} $x$, and the function 
	\begin{align*}
		k:X\times X\to \F, \quad k(x,y)=\ip{k_y,k_x}
	\end{align*}
	is called the \emph{reproducing kernel for} $\Hcal$. Our minimal assumption on the reproducing kernel $k$ is its Borel-measurability. Note that we do not require any form of concentration or decay of $k$. 
	
	By construction, for all $x,y\in X$, 
	\begin{align*}
		k(x,y)=k_y(x)=\overline{\ip{k_x,k_y}}=\overline{k(y,x)}=\overline{k_x(y)}.
	\end{align*}
	Since $\iota : \Hcal \to L^2(X)$ is an isometric embedding, we shall, by the usual abuse of notation, identify $\Hcal$ with its image $\iota(\Hcal)\subset L^2(X)$ and regard elements of $\Hcal$ as (equivalence classes of) functions on $X$. Each $\hf\in \Hcal$ satisfies the pointwise \emph{reproducing formula}
	\begin{align*}
		\hf(x)=\ip{\hf,k_x}=\int_X \hf(y) \,k(x,y) \dmu (y)\quad \text{for $\mu$-a.e.\ } x\in X.
	\end{align*}
	In particular, the reproducing kernel gives rise to a well-defined integral operator 
	\begin{align}\label{eq: integral operator of kernel k}
		K : L^2(X)\to L^2(X), \quad
		\gf \mapsto K\gf := \int_X \gf(y)\,k(\cdot,y)\dmu(y).
	\end{align}
	Indeed, by the Cauchy-Schwarz inequality, $(K\gf)(x) = \ip{\gf,k_x}$ is well-defined for all $x\in X$, and one readily verifies that $K$ is the orthogonal projection onto $\Hcal$.
	
	Let us fix $0\neq \ff \in \Hcal$. We intend to derive a \emph{lower and upper} estimate for the local stability constant $C(\ff)$ in terms of the connectivity of $\ff$ with respect to the kernel $k$.
	For convenience of the reader, we recall that $C(\ff)\in [1,\infty]$ is the smallest constant 
	$C>0$ such that for all $\gf\in \Hcal$,
	\begin{align*}
		\inf_{\alpha \in \T} \norm{\ff-\alpha \gf}_2 \leq C \norm{|\ff|-|\gf|}_2.
	\end{align*}
	For any $S\in \Sigma$ and any $\hf\in L^2(X)$, we define the orthogonal projection of $\hf$ onto the subspace of functions restricted to the set $S$ by $P_S\hf:=\hf \cdot \mathds{1}_S$. 
	
	\begin{definition}\label{def: kernel Cheeger constant of f}
		Let $0\neq\ff\in\Hcal$, and let $K$ be the orthogonal projection onto $\Hcal$ from \eqref{eq: integral operator of kernel k}. The (kernel) Cheeger constant of $\ff$ is defined as 
		\begin{align}\label{eq: def kernel Cheeger const}
			\Cheeger_{\ff}:= \inf_{\substack{S \in \Sigma \\[1mm] 0<\norm{P_S\ff}_2^2\leq\frac{1}{2}\norm{\ff}_2^2}} \frac{\norm{[K,P_S]\ff}_2^2}{\norm{P_S\ff}_2^2} = \inf_{\substack{S \in \Sigma \\[1mm] \norm{P_S\ff}_2, \norm{P_{S^c}\ff}_2>0}} \frac{\norm{[K,P_S]\ff}_2^2}{\norm{P_S\ff}_2^2 \wedge \norm{P_{S^c}\ff}_2^2},
		\end{align}
		assuming at least one such set $S$ exists. Otherwise, we set $\Cheeger_\ff:=1$.
	\end{definition}
	The kernel Cheeger constant $\Cheeger_\ff$ provides a measure for the connectivity of the signal $\ff$ with respect to the kernel. In this form, its value critically depends on the phase of $\ff$. We will later also introduce a variant of this notion of connectivity that only depends on the phaseless measurements $|\ff|$. Corollary~\ref{cor: weighted kernel Cheeger est} below clarifies their relationship.
	
	\begin{remark}\label{rm: wlog F may assumed to be normalized}
		Note that both the local phase retrieval stability constant $C(\ff)$ and the kernel Cheeger constant $\Cheeger_\ff$ are invariant under replacing $\ff$ by $\ff/\norm{\ff}_2$. Furthermore, $\ff$ is phase retrievable if and only if $\ff/\norm{\ff}_2$ is phase retrievable. Thus, it is no loss of generality to assume that $\norm{\ff}_2 = 1$, wherever this may be helpful in the following.
	\end{remark}
	
	Denoting by $K^\perp$ the orthogonal projection onto the orthogonal complement $\Hcal^\perp$ of $\Hcal$ in $L^2(X)$, we have
	\[\norm{[K,P_S]\ff}_2=\norm{KP_S\ff-P_SK\ff}_2=\norm{(K-\id_{L^2(X)})(P_S\ff)}_2=\norm{K^\perp P_S\ff}_2\leq \norm{P_S\ff}_2.\]
	Therefore, $\Cheeger_{\ff}\in [0,1]$ is the largest constant $\Cheeger \in [0,1]$ such that for all $S\in \Sigma$,
	\[\norm{[K,P_S]\ff}_2^2 \geq \Cheeger \, \left(\norm{P_S\ff}_2^2 \wedge \norm{P_{S^c}\ff}_2^2\right).\]
	Moreover, the commutator norm in the numerator of \eqref{eq: def kernel Cheeger const} is symmetric in $S\in \Sigma$, since
	\begin{align*}
		[K,P_S]\ff=[K,\id_{L^2(X)}-P_{S^c}]\ff= [K,\id_{L^2(X)}]\ff-[K,P_{S^c}]\ff=-[K,P_{S^c}]\ff.
	\end{align*}
	This shows that $\norm{[K,P_S]\ff}_2=\norm{[K,P_{S^c}]\ff}_2$ and thereby entails that the infima in~\eqref{eq: def kernel Cheeger const} coincide. Alternative representations for this commutator norm are
	\begin{align}\label{eq: symmetry of commutator norm}
		\norm{[K,P_S]\ff}_2^2
		&= \int_X \left|\int_X \ff(y) \, k(x,y) \cdot \mathds{1}_{\Delta S}(x,y) \dmu(y)\right|^2 \dmu(x) \nonumber\\
		&=\norm{P_SKP_{S^c}\ff}_2^2+\norm{P_{S^c}KP_{S}\ff}_2^2,
	\end{align}
	where $\Delta S:=(S\times S^c) \cup (S^c \times S)$. 
	
	The kernel Cheeger constant has the expected behavior for functions that are approximately supported on disconnected sets.
	\begin{example}\label{ex: kernel Cheeger vs separation by means of shifts}
		Suppose that $X$ is a non-compact SLCA group and assume that $k$ is translation-invariant, i.e. for all $x,y,z\in X$, it holds that
		\[k(x+z,y+z)=k(x,y).\]
		Then, for any \(0\neq \gf,\hf\in\Hcal\), one has
		\[
		\Cheeger_{\gf+T_z\hf}\longrightarrow 0 \qquad \text{as } z\to\infty \text{ in } X,
		\]
		The proof is based on standard \(L^{2}\)-localization arguments, and we therefore defer the details to the appendix; see Proposition~\ref{prop: kernel Cheeger of G+T_zH conv to zero}.
	\end{example}
	
	We now collect auxiliary results that will be used to establish a lower bound for the local stability constant. The underlying idea is to evaluate the defining infimum~\eqref{eq: def kernel Cheeger const} on suitably chosen test functions $\gf_S$ for varying sets $S\in \Sigma$. We define
	\[\gf_S:=K(P_S\ff-P_{S^c}\ff)\in \Hcal.\] 
	First, we derive an upper bound for the \(L^{2}\)-difference of the phaseless measurements in terms of the commutator norm appearing in the numerator of \eqref{eq: def kernel Cheeger const}.
	\begin{lemma}\label{lem: aux lem upper bound denom RKHS}
		For the functions $\ff,\gf_S$ introduced above, we have
		\[\norm{|\ff|-|\gf_S|}_2^2 \leq 4 \norm{[K,P_{S}]\ff}_2^2.\]
	\end{lemma}
	\begin{proof}
		The triangle inequality and the reproducing formula on $\Hcal$ yield the pointwise estimate 
		\begin{align*}
			||\ff|-|\gf_S|| &\leq |\ff-\gf_S| \wedge |\ff+\gf_S|\\
			&= |K(\ff-(P_S\ff-P_{S^c}\ff))|\wedge|K(\ff+(P_S\ff-P_{S^c}\ff))| \\
			&= 2 \left(|KP_{S^c}\ff|\wedge|KP_S\ff|\right).
		\end{align*}
		Employing \eqref{eq: symmetry of commutator norm} and using the bound $|KP_{S^c}\ff|$ on $S$ and $|KP_S\ff|$ on $S^c$ entails
		\begin{align*}
			\norm{|\ff|-|\gf_S|}_2^2 &\leq 4 \, \left(\int_S |(KP_{S^c}\ff)(x)|^2 \dmu(x) + \int_{S^c} |(KP_{S}\ff)(x)|^2 \dmu(x)\right) \\
			&= 4 \, \left(\norm{P_SKP_{S^c}\ff}_2^2+\norm{P_{S^c}KP_{S}\ff}_2^2\right) \\
			&= 4 \norm{[K,P_{S}]\ff}_2^2.
		\end{align*}
	\end{proof}
	The next lemma provides an explicit representation for the distance $\inf_{\alpha \in \T}\norm{\ff-\alpha \, \gf_S}_2$.
	\begin{lemma}\label{lem: aux lem simplified enum RKHS}
		For the functions $\ff,\gf_S$ introduced above, we have
		\[\inf_{\alpha \in \T}\norm{\ff-\alpha \, \gf_S}_2^2=4\,\Bigl(
		\bigl(\|P_S \ff\|_2^2 \wedge \|P_{S^c} \ff\|_2^2\bigr)
		- \|[K,P_S]\ff\|_2^2
		\Bigr)
		.\]
	\end{lemma}
	\begin{proof}
		Writing 
		\[\ff-\alpha\, \gf_S=\ff-\alpha\,(KP_S\ff-KP_{S^c}\ff)=\ff+\alpha(\ff-2KP_S\ff)\] 
		shows that
		\[\norm{\ff-\alpha\, \gf_S}_2^2= \underbrace{\norm{\ff}_2^2+\norm{\ff-2KP_S\ff}_2^2}_{= \const} + 2 \re(\alpha\, \ip{\ff-2KP_S\ff,\ff}).\]
		Moreover, since $K$ is the orthogonal projection onto $\Hcal$, and since $P_S \perp P_{S^c}$, we have
		\begin{align*}
			\ip{\ff-2KP_S\ff,\ff}=\ip{K(\ff-2P_S\ff),K\ff}=\ip{\ff-2P_S\ff,\ff}
			=\norm{P_{S^c}\ff}_2^2-\norm{P_S\ff}_2^2.
		\end{align*}
		Specifically, this calculation shows that $\ip{\ff-2KP_S\ff,\ff}$ is real, implying that
		\[\norm{\ff-\alpha\, \gf_S}_2^2=\const + 2 \re(\alpha) \, \ip{\ff-2KP_S\ff,\ff}.\]
		Therefore, the minimum is attained at $\alpha\in \{\pm 1\}$. Employing the identity $\ff=KP_S\ff+KP_{S^c}\ff$, we arrive at
		\begin{align}\label{eq: identity for d(f,g)}
			\inf_{\alpha \in \T}\norm{\ff-\alpha \, \gf_S}_2^2 = \norm{\ff+\gf_S}_2^2\wedge \norm{\ff-\gf_S}_2^2
			=4 \left(\norm{KP_S\ff}_2^2\wedge \norm{KP_{S^c}\ff}_2^2\right).
		\end{align}
		Furthermore, 
		\[K^\perp P_S\ff=(\id_{L^2(X)}-K)P_S\ff=[P_S,K]\ff,\] 
		which yields
		\begin{align}\label{eq: KP_S vs P_S}
			\norm{KP_S\ff}_2^2=\norm{P_S\ff}_2^2-\norm{K^\perp P_S\ff}_2^2 = \norm{P_S\ff}_2^2 - \norm{[P_S,K]\ff}_2^2.
		\end{align}
		 Recalling the symmetry 
		\[\norm{[K,P_{S}]\ff}_2^2=\norm{[K,P_{S^c}]\ff}_2^2\]
		from \eqref{eq: symmetry of commutator norm}, and inserting \eqref{eq: KP_S vs P_S} into \eqref{eq: identity for d(f,g)} concludes the proof.
	\end{proof}
	
	We now state the postulated lower estimate for the local stability constant $C(\ff)$.
	
	\begin{theorem}\label{thm: RKHS PR lower Cheeger bound}
		Assume that the field of scalars is $\F=\R$ or $\F=\C$.
		For every phase retrievable signal $0\neq \ff \in \Hcal$, one has
		\[
		C(\ff) \geq 
		\sqrt{\Cheeger_{\ff}^{-1} - 1}.
		\]
	\end{theorem}
	
	\begin{proof}
		Recall from Remark~\ref{rm: wlog F may assumed to be normalized} that, without loss of generality, we may assume \(\|\ff\|_{2}=1\).
		Let $S\in \Sigma$ with $0<\norm{P_S\ff}_2^2<\tfrac12$ and consider the test function $\gf = \gf_S\in \Hcal$, as above. (If that does not exist, then the inequality simply reads as $C(\ff)\geq 0$, which is trivial.) Since $\ff$ is phase retrievable, we must have \[\norm{[K,P_{S}]\ff}_2^2>0.\] 
		Otherwise, Lemma~\ref{lem: aux lem upper bound denom RKHS} would imply $|\ff|=|\gf_S|$, hence $\ff\in \T \cdot \gf_S$. However, this would contradict Lemma~\ref{lem: aux lem simplified enum RKHS} and the fact that
		\[\left(\norm{P_S\ff}_2^2\wedge\norm{P_{S^c}\ff}_2^2\right)>0.\] 
		Applying Lemmas~\ref{lem: aux lem upper bound denom RKHS} and \ref{lem: aux lem simplified enum RKHS}, and simplifying the resulting expression, we obtain
		\begin{align}\label{eq: proof kernel Cheeger inequ}
		\inf_{\alpha \in \T} \| \ff - \alpha \gf_S \|_2^2
		\geq 
		\left(\frac{\norm{P_S\ff}_2^2\wedge\norm{P_{S^c}\ff}_2^2}
		{\norm{[K,P_{S}]\ff}_2^2} - 1\right)\, \norm{|\ff| - |\gf_S|}_2^2.
		\end{align}
		Recalling both the definition of the local stability constant $C(\ff)$ and the kernel Cheeger constant $\Cheeger_{\ff}$, we conclude the desired estimate.
	\end{proof}
	Next, we introduce a variant of the kernel Cheeger constant $\Cheeger_\ff$ that only depends on the phaseless measurements of $|\ff|$ and compare this new measure of connectivity with $\Cheeger_\ff$. For a measurable and symmetric weight $\omega:X\times X \to \R_{\geq 0}$, we introduce the uniform $L^1$-bound
	\begin{align*}
		\norm{\omega}_{L^\infty(L^1)}:=\sup_{y\in X} \int_X |\omega(x,y)|\dmu (x) <\infty.
	\end{align*}
	Moreover, for a function $\hf\in L^2(X)$, we also define its weighted norm
	\begin{align}\label{eq: def of 2,w norm and embedding}
		\norm{\hf}_{2,\omega}^2:=\int_{X}\int_{X} |\hf(y)|^2\, \omega(x,y) \dmu(x)\dmu(y)\leq \norm{\omega}_{L^\infty(L^1)} \norm{\hf}_2^2 <\infty.
	\end{align}
	These definitions naturally give rise to the \emph{weighted kernel Cheeger constant}
	\[\Cheeger_{\ff,\omega}:=\inf_{\substack{S \in \Sigma \\[1mm] \norm{P_S\ff}_{2,\omega}^2,\norm{P_{S^c}\ff}_{2,\omega}^2 >0}} \frac{\int_X \int_X |\ff(y)|^2 \, \omega(x,y) \cdot \mathds{1}_{\Delta S}(x,y) \dmu(y) \dmu(x)}{\norm{P_S\ff}_{2,\omega}^2\wedge \norm{P_{S^c}\ff}_{2,\omega}^2},\]
	where $\Delta S=(S\times S^c) \cup (S^c \times S)$, assuming at least one such set $S$ exists and $\Cheeger_{\ff,\omega}:=\norm{\omega}_{L^\infty(L^1)}^{-2}$ otherwise. 
	\begin{corollary}\label{cor: weighted kernel Cheeger est}
		Assume that the field of scalars is $\F=\R$ or $\F=\C$.
		Let $\omega:X\times X \to \R_{\geq 0}$ be a measurable, symmetric weight satisfying the uniform $L^1$-bound $\norm{\omega}_{L^\infty(L^1)}<\infty$, and assume that $|k|\leq \omega$ $\mu$-a.e. 
		
		Then, for every signal $0\neq \ff \in \Hcal$,
		\[
		\Cheeger_\ff\leq \norm{\omega}_{L^\infty(L^1)}^2 \Cheeger_{\ff,\omega}.
		\]

		In particular, if $\ff$ is phase retrievable, one has
		\[
		C(\ff) \geq 
		\sqrt{\norm{\omega}_{L^\infty(L^1)}^{-2} \Cheeger_{\ff,\omega}^{-1} \,- \,1}.
		\]
	\end{corollary}
	\begin{proof}
		Let $S\in \Sigma$. On the one hand, \eqref{eq: symmetry of commutator norm} yields
		\begin{align*}
			\norm{[K,P_S]\ff}_2^2 &= \int_X \left|\int_X \ff(y) \, k(x,y) \cdot \mathds{1}_{\Delta S}(x,y) \dmu(y)\right|^2 \dmu(x) \\
			&\leq \int_X \left(\int_X |\ff(y)| \, \omega(x,y) \cdot \mathds{1}_{\Delta S}(x,y) \dmu(y)\right)^2 \dmu(x) \\
			&\leq \int_X \left(\int_X |\ff(y)|^2 \, \omega(x,y) \cdot \mathds{1}_{\Delta S}(x,y) \dmu(y)\right) \, \left(\int_X \omega(x,y) \dmu(y)\right) \dmu(x) \\
			&\leq \norm{\omega}_{L^\infty(L^1)} \, \int_X \int_X |\ff(y)|^2 \, \omega(x,y) \cdot \mathds{1}_{\Delta S}(x,y) \dmu(y) \dmu(x).
		\end{align*}
		On the other, \eqref{eq: def of 2,w norm and embedding} entails
		\begin{align*}
			\left(\norm{P_S\ff}_2^2\wedge\norm{P_{S^c}\ff}_2^2\right) \geq 	\norm{\omega}_{L^\infty(L^1)}^{-1}\, \left(\norm{P_S\ff}_{2,\omega}^2\wedge \norm{P_{S^c}\ff}_{2,\omega}^2\right).
		\end{align*}		
		Combining these inequalities with the definitions of $\Cheeger_\ff$ and $\Cheeger_{\ff,\omega}$ concludes the proof.
	\end{proof}
	\begin{remark}
		Observe that the assumptions of Corollary~\ref{cor: weighted kernel Cheeger est} impose the uniform $L^1$-bound $\norm{k}_{L^\infty(L^1)}<\infty$ on the reproducing kernel. Specifically, using $\omega=|k|$ yields, for every phase retrievable signal $0\neq \ff\in \Hcal$,
		\[
		C(\ff) \geq 
		\sqrt{\norm{k}_{L^\infty(L^1)}^{-2} \Cheeger_{\ff,|k|}^{-1} \,- \,1}.
		\]
	\end{remark}
	
	Let us now turn our attention towards proving an upper bound on the stability constant $C(\ff)$ under the additional assumption that the scalar field $\F$ is $\R$.
	The proof of our first lemma in this direction follows ideas from \cite{cheng2021stable}. 
	
	\begin{lemma}\label{lem: real RKHS commutator norm vs measurements distance}
		Assume that the field of scalars is $\F=\R$. Let $\ff,\hf\in \Hcal$, and define
		\[S:=\{x\in X\,|\, \ff(x)\hf(x)\geq 0\} \in \Sigma.\]
		Then, we have 
		\[\norm{[K,P_S]\ff}_2\leq \norm{|\ff|-|\hf|}_2.\]
	\end{lemma}
	\begin{proof}
		We define $\zeta_S:=\mathds{1}_S-\mathds{1}_{S^c}$ and note that for all $x,y\in X$,
		\[\mathds{1}_{X\times S}(x,y) - \mathds{1}_{S\times X}(x,y)=\frac{1}{2} (\zeta_S(y)-\zeta_S(x)).\]
		A direct calculation shows that
		\begin{align*}
			\norm{[K,P_S]\ff}_2^2 &= \int_X \left|(KP_S\ff)(x)-(P_SK\ff)(x)\right|^2 \dmu(x) \nonumber\\
			&= \int_X \left|\int_X \ff(y) \, k(x,y) \, (\mathds{1}_{X\times S}(x,y)- \mathds{1}_{S\times X}(x,y)) \dmu(y)\right|^2 \dmu(x) \nonumber\\
			&= \frac{1}{4}\int_X \left|\int_X \ff(y) \, k(x,y) \, (\zeta_S(y)-\zeta_S(x))\dmu(y)\right|^2 \dmu(x) \nonumber\\
			&=\frac{1}{4} \norm{K(\zeta_S \ff)-\zeta_S \ff}_2^2.
		\end{align*}
		Inserting $\pm \hf$ and recalling the definition of $\zeta_S$ (respectively $S$), 
		we find that
		\begin{align}\label{eq: h vs zeta_S f}
			\norm{K(\zeta_S \ff)-\zeta_S \ff}_2^2 
			&\leq 2(\norm{K(\zeta_S \ff)-\hf}_2^2+\norm{\hf-\zeta_S \ff}_2^2) \nonumber\\
			&\leq 4 \norm{\zeta_S \ff-\hf}_2^2 \nonumber\\
			&=4\norm{|\ff|-|\hf|}_2^2.
		\end{align}
	\end{proof}
	Having established the necessary ingredients, we now state and prove the following general stability result.
	\begin{theorem}\label{thm: RKHS SR upper bound}
		Assume that the field of scalars is $\F=\R$. For every signal $0\neq \ff\in \Hcal$, one has
		\[C(\ff) \leq 2\,\sqrt{\Cheeger_\ff^{-1}-1} +1.\]
		In particular, if $\Cheeger_\ff>0$ then $\ff$ does locally stable sign retrieval.
	\end{theorem}
	\begin{proof}
		Fix any $\hf\in \Hcal$.
		As before, we consider the set $S:=\{x\in X\,|\, \ff(x)\hf(x)\geq 0\} \in \Sigma\,$ and define
		$\zeta_S:=\mathds{1}_S-\mathds{1}_{S^c}$ as well as the test function
		\[\gf_S:=K(P_S\ff-P_{S^c}\ff)=K(\zeta_S \ff)\in \Hcal.\]
		Similar to \eqref{eq: h vs zeta_S f}, we have 
		\begin{align*}
			\left(\norm{\ff -\hf}_2\wedge \norm{\ff +\hf}_2\right)&\leq \left(\norm{\ff - \gf_S}_2 \wedge \norm{\ff + \gf_S}_2\right) + \norm{\gf_S-\hf}_2 \\
			&\leq  \left(\norm{\ff - \gf_S}_2 \wedge \norm{\ff + \gf_S}_2\right) + \norm{|\ff|-|\hf|}_2.
		\end{align*}
		Combining Lemma~\ref{lem: aux lem simplified enum RKHS} with Lemma~\ref{lem: real RKHS commutator norm vs measurements distance} and recalling the definition of the kernel Cheeger constant yields
		\begin{align*}
			\left(\norm{\ff - \gf_S}_2^2 \wedge \norm{\ff + \gf_S}_2^2\right)
			&=4\, \Bigl( \bigl(\norm{P_S\ff}_2^2\wedge\norm{P_{S^c}\ff}_2^2\bigr)-\norm{[K,P_{S}]\ff}_2^2\Bigr) \\
			&\leq 4 \,\bigl(\Cheeger_\ff^{-1}-1\bigr) \norm{[K,P_S]\ff}_2^2 \\
			&\leq 4 \,\bigl(\Cheeger_\ff^{-1}-1\bigr) \norm{|\ff|-|\hf|}_2^2.
		\end{align*}
		Altogether, it follows that
		\begin{align*}
			\left(\norm{\ff -\hf}_2\wedge \norm{\ff +\hf}_2\right)
			\leq \Bigl(2\,\sqrt{\Cheeger_\ff^{-1}-1} +1\Bigr)\,\norm{|\ff|-|\hf|}_2,
		\end{align*}
		which concludes the proof of the theorem.
	\end{proof}
	
	\begin{remark}
		It is natural to expect that analogues of the results in this section extend beyond the Hilbertian setting, for instance to reproducing kernel Banach spaces realized as subspaces of (weighted) \(L^{p}\)-spaces. Such extensions may, however, require additional structural assumptions on the kernel \(K\), and we do not pursue them here. 
	\end{remark}

	\section{From Kernel to Graph Cheeger Constants in the (Semi-)Discrete Setting}\label{sec: Cheeger const semi-discrete}
	
	In this section we return to the convolutional phase retrieval problem arising from generalized wavelet transforms, as discussed in Section~\ref{sec: PR conv}. Here, we additionally assume that the index set $\Lambda$ is countable and discrete, equipped with the counting measure $\nu$. The system $\Psi$ is assumed to satisfy the Calder\'on condition \textup{(C)}. 
	
	We show that the kernel Cheeger constant is closely related to the classical graph Cheeger constant of a suitably defined data-dependent weighted graph. To see this, first recall from Example \ref{ex: coherent state systems vs RKHS} that $\Hcal_\Psi:=W_\Psi(L^2(A))\subset L^2(A\times \Lambda)$ is an RKHS with associated orthogonal projection $K_\Psi=W_\Psi W_\Psi^\ast$ and reproducing kernel $k_\Psi$, for $((x,\lambda),(x^\prime,\lambda^\prime))\in (A\times \Lambda)^2$ given by
	\begin{align*}
		k_\Psi((x,\lambda),(x^\prime,\lambda^\prime)):=(\psi_{\lambda^\prime} * \psi_{\lambda}^*)(x-x^\prime)=\ip{T_{x^\prime}\psi_{\lambda^\prime},T_x\psi_\lambda}.
	\end{align*}
	Let us fix $0\neq f \in L^2(A)$. We intend to derive a \emph{lower} estimate for the local stability constant $C(f):=C(W_\Psi f)$ in terms of the connectivity properties of a suitable, data-dependent graph.
	Recalling from Theorem~\ref{thm: characterization of properties C and P} that the wavelet transform $W_\Psi$ is an isometry, the local stability constant $C(f)\in[1,\infty]$ is---in accordance with the definition from \eqref{def: locally stable PR}---the smallest constant $C>0$ such that for all $g\in L^2(A)$,
	\begin{equation*}
		\inf_{\alpha \in \T} \| f - \alpha g \|_{L^2(A)} \leq C \norm{| W_\Psi f | - |W_\Psi g| }_{L^2(A\times \Lambda)}.
	\end{equation*}
	We define a weighted graph $G=(V,E,w)$, where
	\begin{align}\label{eq: def of V and E}
		V = \{ \lambda \in \Lambda : \| f \ast \psi_\lambda \|_2 > 0 \} \quad \text{and} \quad E = \{ (\lambda,\lambda') \in V \times V\,:\, \| f \ast \psi_\lambda \ast \psi_{\lambda'} \|_2 > 0 \}.
	\end{align} 
	The associated weights are given by 
	\begin{align}\label{eq: def of graph weights}
		w_\lambda = \| f \ast \psi_\lambda \|_2^2 \quad (\lambda \in V)\,, \qquad  w_{\lambda,\lambda'} = \| f \ast \psi_\lambda \ast \psi_{\lambda'} \|_2^2 \quad ((\lambda,\lambda^\prime)\in E).
	\end{align}
	Observe that all quantities depend on $f$, even though we do not make this explicit in the notation. In view of the notation introduced in Definition~\ref{def: approx_f and sim_f}, $G=G(f)$ is the natural weighted symmetric graph induced by the relation $\approx_f$.
	
	We intend to compare the local stability constant $C(f)$ with the graph Cheeger constant of $G=G(f)$, defined next. Given any subset $S \subset V$, we let $S^c = V \setminus S$, and define its boundary 
	\[
	\partial S:=\{(\lambda,\lambda^\prime)\in E\,|\, \lambda\in S\text{ and }\lambda^\prime\in S^c\}.
	\]
	Note that $(\lambda,\lambda^\prime)\in\partial S$ if and only if $(\lambda^\prime,\lambda)\in\partial S^c$. Since the weights are symmetric, this entails 
	\[\sum_{(\lambda,\lambda') \in \partial S} w_{\lambda,\lambda'}=\sum_{(\lambda,\lambda^\prime) \in \partial S^c} w_{\lambda,\lambda'}.\]
	\begin{definition}\label{def: Cheeger constant of G(f)}
		Let $0\neq f\in L^2(A)$, and let $G(f)$ be the weighted graph induced by $f$ as described in \eqref{eq: def of V and E} and \eqref{eq: def of graph weights}. For $\# V>1$, the (graph) Cheeger constant of $G(f)$ is defined as 
		\begin{align}\label{def: graph Cheeger const}
			\Cheeger_{G(f)}:= \inf_{\substack{S \subset V \\ 0<\sum_{\lambda\in S}w_\lambda\leq\frac{1}{2}\norm{f}_2^2}} \frac{\sum_{(\lambda,\lambda') \in \partial S} w_{\lambda,\lambda'}}{\sum_{\lambda \in S} w_\lambda} = \inf_{\emptyset \not= S \subsetneq V} \frac{\sum_{(\lambda,\lambda') \in \partial S} w_{\lambda,\lambda'}}{\left(\sum_{\lambda \in S} w_\lambda\right) \wedge \left(\sum_{\lambda \in S^c} w_\lambda \right)}.
		\end{align}
		For $\# V=1$, we set $\Cheeger_{G(f)}:=1$.
	\end{definition}
	Note that the Calder\'on condition \textup{(C)} and Theorem~\ref{thm: characterization of properties C and P} imply that 
	\begin{align*}
		\sum_{\lambda \in V} w_\lambda = \norm{f}_2^2 \quad \text{and} \quad \forall \lambda' \in V\,:\, \sum_{\lambda \in V} w_{\lambda,\lambda'} = w_{\lambda'}.
	\end{align*} 
	Specifically, this entails that $\Cheeger_{G(f)}\leq 1$. Therefore, $\Cheeger_{G(f)}\in [0,1]$ is the largest constant $\Cheeger \in [0,1]$ such that for all $S \subset V$,
	\[\sum_{(\lambda,\lambda') \in \partial S} w_{\lambda,\lambda'}\geq \Cheeger \, \biggl(\Bigl(\sum_{\lambda \in S} w_\lambda\Bigr) \wedge \Bigl(\sum_{\lambda \in S^c} w_\lambda \Bigr)\biggr).\]
	\begin{remark}\label{rm: wlog f may assumed to be normalized}
		In analogy to Remark~\ref{rm: wlog F may assumed to be normalized}, we note that both the local phase retrieval stability constant $C(f)$ and the graph Cheeger constant $\Cheeger_{G(f)}$ are invariant under replacing $f$ by $f/\norm{f}_2$.
	\end{remark}
	The Cheeger constant provides a measure for how well connected the weighted graph $G(f)$ is. For convenience of the reader, we recall its relationship with the algebraic connectivity $\Afrak_{G(f)}$ of the graph $G=G(f)$, briefly summarizing the discussion in \cite[Section 2.5]{cheng2021stable}. For more details and background, we refer the reader to \cite{chung1996laplacians,cheeger2015lower,cheng2021stable}.
	\begin{remark}\label{rm: alg connect and Cheeger inequ}
		For any graph signal $z=(z_\lambda)_{\lambda\in V} \in \ell^2(V;(w_\lambda)_{\lambda\in V})$, define
		\[
		\Lcal_{G(f)}(z):=\sum_{(\lambda,\lambda^\prime)\in E}  w_{\lambda,\lambda^\prime} \, |z_\lambda-z_{\lambda^\prime}|^2.
		\]
		The algebraic connectivity $\Afrak_{G(f)}$ of $G(f)$ is
		\begin{align}\label{eq: alg connect}
			\Afrak_{G(f)}:=\inf_{z\neq 0, \ip{z,\mathds{1}}=0} 
			\frac{\Lcal_{G(f)}(z)}{\sum_{\lambda\in V} w_\lambda \, |z_\lambda|^2},
		\end{align}
		where $\ip{\,,\,}$ denotes the natural weighted inner product in $\ell^2(V;(w_\lambda)_{\lambda\in V})$, i.e.,
		\[
		\ip{z,\mathds{1}}=\sum_{\lambda\in V} z_\lambda \, w_{\lambda}.
		\]
		The algebraic connectivity is closely related to the minimum of the spectrum of the (normalized) graph Laplacian corresponding to $G$ on a suitable subspace of $\ell^2(V;(w_\lambda)_{\lambda\in V})$. If the graph is finite, the algebraic connectivity is strictly positive if and only if it is connected. If the graph is countably infinite, merely the only if part of this statement remains valid in general. In particular, if $G$ is disconnected then $\Afrak_{G(f)}=0$. 
		
		Denoting by $D_f:=\sup_{\lambda\in V} \#[\lambda]_{\approx_f}$ the (maximal) degree of $G(f)$, the Cheeger inequality reads
		\begin{align}\label{eq: graph Cheeger inequality}
			2 \Cheeger_{G(f)} \geq \Afrak_{G(f)} \geq \frac{\Cheeger_{G(f)}^2}{2D_f}.
		\end{align}
	\end{remark}
	The link between the local stability constant for the phase retrieval problem and the graph Cheeger constant is indirectly obtained by relating the kernel Cheeger constant $\Cheeger_{W_\Psi f}$ from the previous section with the graph Cheeger constant $\Cheeger_{G(f)}$. 
	\begin{lemma} \label{lem: properties of H_S}
		Let $V$ be the vertex set of $G(f)$ as defined in~\eqref{eq: def of V and E}. Then, for any subset $S\subset V$, the auxiliary function 
		\[
		H_S : \widehat{A} \to \R_{\geq 0}, \quad H_S(\xi) = \sum_{\lambda \in S} |\widehat{\psi}_\lambda(\xi)|^2 
		\]
		is measurable and a.e. bounded through $0 \le H_S \le 1$. Moreover, for all $h \in L^2(A)$,
		\begin{equation} \label{eqn:transfer}
			\sum_{\lambda \in S} \| h \ast \psi_\lambda \|_2^2 = \int_{\widehat{A}} |\widehat{h}(\xi)|^2  H_S(\xi) \dxi. 
		\end{equation}
		Finally, $H_S$ satisfies
		\[
		H_S(\xi) + H_{S^c}(\xi) = 1, \quad \text{a.e. } \xi \in\{\widehat{f}\neq 0\}, 
		\]
		as well as
		\[H_S(\xi)\,H_{S^c}(\xi)^2+H_S(\xi)^2\,H_{S^c}(\xi)=H_S(\xi)\,H_{S^c}(\xi), \quad \text{a.e. } \xi \in\{\widehat{f}\neq 0\}.\]
	\end{lemma}
	\begin{proof}
		The Calder\'on condition \textup{(C)} immediately entails that $0 \le H_S \le 1$ a.e. Clearly, $H_S$ is also measurable. Equation~\eqref{eqn:transfer} is obtained by a straightforward combination of Plancherel and convolution theorem, together with Tonelli's theorem. 
		
		For the remaining two identities, note that the second follows immediately from the first. To prove the first, it suffices to show that
		\begin{align}\label{eq: identity for H_S}
			|\widehat{f}(\xi)|^2 \, (1-	H_S(\xi) - H_{S^c}(\xi))=0, \quad \text{a.e. } \xi \in\{\widehat{f}\neq 0\}. 
		\end{align}
		Combining again the Calder\'on condition \textup{(C)} with the Plancherel theorem yields
		\begin{eqnarray*}
			\int_{\{\widehat{f}\neq 0\}} |\widehat{f}(\xi)|^2 \dxi
			& = & \int_{\{\widehat{f}\neq 0\}} |\widehat{f}(\xi)|^2 \sum_{\lambda \in \Lambda} |\widehat{\psi}_\lambda(\xi)|^2 \dxi \\
			& = & \sum_{\lambda \in \Lambda} \int_{\{\widehat{f}\neq 0\}} |\widehat{f}(\xi)|^2 |\widehat{\psi}_\lambda(\xi)|^2 \dxi \\
			& = & \sum_{\lambda \in \Lambda} \| f \ast \psi_\lambda \|_2^2\\
			& = & \sum_{\lambda \in V} \| f \ast \psi_\lambda \|_2^2\\
			& = & \int_{\{\widehat{f}\neq 0\}} |\widehat{f}(\xi)|^2 \sum_{\lambda \in V} |\widehat{\psi}_\lambda(\xi)|^2 \dxi \\
			& = &  \int_{\{\widehat{f}\neq 0\}} |\widehat{f}(\xi)|^2 (H_S(\xi) + H_{S^c}(\xi)) \dxi.
		\end{eqnarray*}
		This concludes \eqref{eq: identity for H_S} and hence the proof of this lemma.
	\end{proof}
	
	The next lemma provides the key observation linking the different notions of Cheeger constants. We emphasize that, in the present context, the relevant object for the application of the results from Section~\ref{sec: cheeger RKHS} is the wavelet transform $F:=W_\Psi f\in \Hcal_\Psi$, rather than $f$ itself.
	\begin{lemma}\label{lem: comm norm [K,P_{AxS}] vs graph weights}
		Let the notation be as above. For any subset $S\subset V$, we have 
		\begin{align*}
			\norm{[K_\Psi,P_{A\times S}] W_\Psi f}_{L^2(A\times \Lambda)}^2 = \sum_{(\lambda,\lambda') \in \partial S} w_{\lambda,\lambda'}
		\end{align*}
		and 
		\begin{align*}
			\norm{P_{A\times S} W_\Psi f}_{L^2(A\times \Lambda)}^2 
			= \sum_{\lambda\in S} w_\lambda.
		\end{align*}
	\end{lemma}
	\begin{proof}
		A direct calculation yields
		\begin{align*}
			K_\Psi P_{A\times S}W_\Psi f(x,\lambda)= \sum_{\lambda^\prime \in S} (f * \psi_\lambda^\ast * \psi_{\lambda^\prime}^\ast * \psi_{\lambda^\prime})(x)
		\end{align*}
		as well as 
		\begin{align*}
			P_{A\times S}K_\Psi W_\Psi f(x,\lambda) = \begin{cases}
				\sum_{\lambda^\prime \in V} (f * \psi_\lambda^\ast * \psi_{\lambda^\prime}^\ast * \psi_{\lambda^\prime})(x), &\lambda\in S \\
				0, &\lambda\in \Lambda\setminus S
			\end{cases}.
		\end{align*}
		Using the Plancherel theorem, the convolution theorem as well as Tonelli's theorem, it follows that
		\begin{eqnarray*}
			\norm{[K_\Psi,P_{A\times S}] W_\Psi f}_2^2 &= &\sum_{\lambda\in S} \norm{\sum_{\lambda^\prime \in S^c} f * \psi_\lambda^\ast * \psi_{\lambda^\prime}^\ast * \psi_{\lambda^\prime}}_{2}^2 + \sum_{\lambda\in S^c} \norm{\sum_{\lambda^\prime \in S} f * \psi_\lambda^\ast * \psi_{\lambda^\prime}^\ast * \psi_{\lambda^\prime}}_{2}^2\\
			& = & \sum_{\lambda \in S} \int_{\widehat{A}} |\widehat{f}(\xi)|^2 \, |\widehat{\psi}_{\lambda}(\xi)|^2 \left( \sum_{\lambda^\prime \in S^c} |\widehat{\psi}_{\lambda^\prime}(\xi)|^2 \right)^2 \dxi \\
			&\qquad + &\sum_{\lambda \in S^c} \int_{\widehat{A}} |\widehat{f}(\xi)|^2 \, |\widehat{\psi}_{\lambda}(\xi)|^2 \left( \sum_{\lambda^\prime \in S} |\widehat{\psi}_{\lambda^\prime}(\xi)|^2 \right)^2 \dxi\\
			& = & \int_{\widehat{A}} |\widehat{f}(\xi)|^2 \, (H_S(\xi)\,H_{S^c}(\xi)^2+H_S(\xi)^2\,H_{S^c}(\xi)) \dxi. 
		\end{eqnarray*}
		Applying Lemma~\ref{lem: properties of H_S} and recalling the definition of the weights, we conclude that 
		\[\norm{[K_\Psi,P_{A\times S}] W_\Psi f}_2^2 = \int_{\widehat{A}} |\widehat{f}(\xi)|^2 \, H_S(\xi) \,H_{S^c}(\xi)  \dxi = \sum_{(\lambda,\lambda') \in \partial S} w_{\lambda,\lambda'}.\]
		The second equation of the lemma follows immediately from the definition of the weights, since
		\begin{align*}
			\norm{P_{A\times S} W_\Psi f}_{L^2(A\times \Lambda)}^2 = \sum_{\lambda\in S} \norm{f*\psi_\lambda^*}_{L^2(A)}^2
			= \sum_{\lambda\in S} w_\lambda.
		\end{align*}
	\end{proof}
	
	We are now in a position to state and prove the main comparison between the kernel Cheeger constant $\Cheeger_{W_\Psi f}$ from Section~\ref{sec: cheeger RKHS} and the graph Cheeger constant $\Cheeger_{G(f)}$ introduced at the beginning of this section.
	\begin{theorem}\label{thm: kernel vs. graph Cheeger constant semi-discrete}
		Assume that the field of scalars is $\F=\R$ or $\F=\C$. 
		For every signal $0\neq f \in L^2(A)$, one has 
		\[\Cheeger_{W_\Psi f}\leq \Cheeger_{G(f)}.\]
		In particular, if $f$ is phase retrievable, then
		\[
		C(f) \geq 
		\sqrt{\Cheeger_{G(f)}^{-1} - 1}.
		\]
	\end{theorem}
	\begin{proof}
		Restricting the infimum defining the kernel Cheeger constant $\Cheeger_{W_\Psi f}$ to product sets $A\times S$ and invoking Lemma~\ref{lem: comm norm [K,P_{AxS}] vs graph weights} yields the claim.
	\end{proof}
	
	The argument of \cite[Theorem 2.9]{cheng2021stable} can be adapted to our setting to obtain an upper bound for the local stability constant $C(f)$ even when the field of scalars is $\F=\C$. Let $V$ and $E$ be as in \eqref{eq: def of V and E}, and define
	\[
	m_\lambda := |f*\psi_\lambda^\ast| \in L^2(A) \quad (\lambda\in V)\,, 
	\qquad
	m_{\lambda,\lambda'} := |f*\psi_\lambda^\ast*\psi_{\lambda'}^\ast| \in L^2(A) \quad ((\lambda,\lambda')\in E).
	\]
	Introduce the temporal weighted graph
	\[
	G_t(f):= (V_x,E_x,(m_\lambda(x))_{\lambda\in V},(m_{\lambda,\lambda'}(x))_{(\lambda,\lambda')\in E})_{x\in A},
	\]
	where
	\[
	V_x := \{\lambda\in V : |(f*\psi_\lambda)(x)|>0\},
	\qquad
	E_x := \{(\lambda,\lambda')\in E : |(f*\psi_\lambda*\psi_{\lambda'})(x)|>0\}.
	\]
	For any temporal graph signal $z=(z_\lambda)_{\lambda\in V}\in L^2(A\times V;|W_\Psi f|^2)$, set
	\[
	\mathcal{L}_{G_t(f)}(z):= \sum_{(\lambda,\lambda^\prime)\in E} \norm{(z_\lambda-z_{\lambda^\prime})\, m_{\lambda,\lambda^\prime}}_2^2.
	\]
	In analogy to Remark~\ref{rm: alg connect and Cheeger inequ}, define the algebraic connectivity of $G_t(f)$ by
	\begin{equation}\label{eq: alg connect of temp graph}
		\mathfrak{A}_{G_t(f)} := 
		\inf_{\substack{0\neq z\in L^2(A\times V;|W_\Psi f|^2)\\ \langle z,\mathds{1}\rangle=0}}
		\frac{\Lcal_{G_t(f)}(z)}{\sum_{\lambda\in V}\norm{z_\lambda \, m_\lambda}_2^2},
	\end{equation}
	where $\langle\,,\,\rangle$ denotes the natural weighted inner product on $L^2(A\times V;|W_\Psi f|^2)$, i.e.
	\[
	\langle z,\mathds{1}\rangle=\sum_{\lambda\in V}\int_A z_\lambda(x)\, m_\lambda(x)^2\,\dx .
	\]
	Restricting the infimum in \eqref{eq: alg connect of temp graph} to those $z$ for which each $z_\lambda$ is constant on $A$ immediately yields
	\begin{equation}\label{eq: alg connect G(f) vs G_t(f)}
		0 \le \mathfrak{A}_{G_t(f)} \le \mathfrak{A}_{G(f)},
	\end{equation}
	where $\mathfrak{A}_{G(f)}$ is the (usual) algebraic connectivity of $G(f)$ defined in \eqref{eq: alg connect}. This formalizes the intuition that $\mathfrak{A}_{G(f)}$ measures how well the temporal graph family---determined by the interplay of $f$ and the filters---is connected (on average over $A$), whereas $\mathfrak{A}_{G_t(f)}$ captures the local connectivity of $f$ with respect to the filters on $A$. In particular, if $G(f)$ is disconnected, then $\mathfrak{A}_{G(f)}=\mathfrak{A}_{G_t(f)}=0$.
	
	With these adjustments, the proof of \cite{cheng2021stable} carries over almost verbatim to our setting. In particular, one obtains the following upper bound for the local stability constant of the underlying phase retrieval problem.
	\begin{proposition}
		Assume that the field of scalars is $\F=\R$ or $\F=\C$. 
		For every signal $0\neq f \in L^2(A)$, one has 
		\begin{equation}\label{eq: analog upper bound on C(f) through temp alg connect}
			C(f) \leq \sqrt{32\, D_f\, \Afrak_{G_t(f)}^{-1} + 10},
		\end{equation}
		where $D_f=\sup_{\lambda\in V}\# [\lambda]_{\approx_f}$ denotes the (maximal) degree of $G(f)$.
	\end{proposition}
	We note, however, that \eqref{eq: analog upper bound on C(f) through temp alg connect} is only informative when $D_f<\infty$ and $\mathfrak{A}_{G_t(f)}>0$. The first condition is enforced if $D_\Psi:=\sup_{\lambda\in \Lambda}\# \{\lambda^\prime\in \Lambda\,|\, \norm{\psi_\lambda*\psi_{\lambda^\prime}}_2>0\}$ is finite, which is a standard assumption on the filterbank $\Psi$, whereas the second appears to be fairly restrictive in our setting. While replacing $\mathfrak{A}_{G_t(f)}$ by $\mathfrak{A}_{G(f)}$ would substantially relax the assumption, we do not expect such a bound to hold in full generality: by the graph Cheeger inequality~\eqref{eq: graph Cheeger inequality}, an estimate of the form $C(f)\lesssim \mathfrak{A}_{G(f)}^{-1/2}$ would imply $C(f)\lesssim \Cheeger_{G(f)}^{-1}$, yet Example~\ref{ex: Cheeger of f+T_x f} shows that this cannot generally be true globally for signals in $L^2(A)$. 

	\begin{example}\label{ex: Cheeger of f+T_x f}
		Assume that $A$ is non-compact and let $0\neq h\in L^2(A)$ be such that the graph $G(h)$ is finite and connected (for instance, $h=\psi_{\lambda^\ast}$ for any $\lambda^\ast\in \Lambda$ if $D_\Psi<\infty$). After normalization, we may assume that $\norm{h}_2=1$. 
		
		Let $\varepsilon>0$. 
		Since $W_\Psi$ is an isometry, there exists a compact set $Q=Q(\varepsilon)\subset A$ such that 
		\[
		\int_{Q} \sum_{\lambda\in \Lambda} |(h*\psi_\lambda^\ast)(y)|^2 \, \dy 
		= \int_{Q} \sum_{\lambda\in \Lambda} |W_\Psi h(y,\lambda)|^2 \, \dy 
		\geq 1-\varepsilon.
		\]
		Using this concentration estimate, one readily checks that, for all $x\in A\setminus (Q-Q)$,
		\[
		\norm{\,|W_\Psi(h\pm T_xh)| - |W_\Psi h| - T_x|W_\Psi h|\,}_{L^2(A\times \Lambda)} \lesssim \varepsilon.
		\]
		In other words, up to an error controlled by $\varepsilon$, the phaseless measurements agree in $L^2$-norm:
		\[
		|W_\Psi(h+T_xh)| \approx |W_\Psi h| + T_x|W_\Psi h| \approx |W_\Psi(h-T_xh)|.
		\]
		At the same time, we obtain that, for any $\alpha\in \T$, 
		\[
		2 - \norm{W_\Psi(h+ T_xh)-\alpha \, W_\Psi(h- T_xh)}_{L^2(A\times \Lambda)} \lesssim \varepsilon.
		\]
		Thus,
		\[
		C(h+T_xh)\gtrsim \frac{2-\varepsilon}{\varepsilon}.
		\]
		
		In particular, this example shows that the local stability constant for the generalized wavelet phase retrieval problem can become arbitrarily large within the signal class 
		\[\{h+T_xh : x\in A\}\subset L^2(A).\] 
		This type of construction, based on time-frequency (or, more generally, component) separation, is standard and has already been used in various settings; see, e.g., \cite{cahill2016PR,alaifari2017PRcont,alaifari2021gabor}.
		
		On the other hand, we can choose $Q$ sufficiently large so that 
		\[
		\Cheeger_{G(h+T_xh)} \geq \frac{\Cheeger_{G(h)}}{4}>0.
		\]
		In particular, the local phase retrieval stability constant cannot, in general, be controlled by the inverse of the graph Cheeger constant. The graph Cheeger constant being bounded away from zero for spatially separated signals of this type is in stark contrast to the behavior of the kernel Cheeger constant. We provide the details for these claims in the appendix; see Proposition~\ref{prop: aux lem Cheeger const h vs h+T_xh}.
	\end{example}

\section{Appendix}\label{sec: app}

In this appendix we collect the proofs deferred from the main text.

We begin with the proof for an earlier claim from Remark~\ref{rm: conditions (C) and (SI)} in Section~\ref{sec: PR conv}. Here, we are in the classical setting of the one-dimensional wavelet transform, i.e., $A=\R$ and $\Lambda=H=\R^\ast$ in the terminology of the same remark.
\begin{proposition}\label{prop: sufficient condition for (SI)}
	Fix $0\neq \psi \in L^2(\R)$. 
	Let $\psi_{\lambda}(x)=|\lambda|^{-1} \psi(\lambda^{-1}x)$, $\lambda\in \R^\ast$. If $\psi$ has compact support $\supp(\widehat{\psi})\subset \R_{> 0}$, then $\Psi=(\psi_\lambda)_{\lambda\in \R^\ast}$ satisfies the spectral injectivity condition \textup{(SI)}.
	
	In particular, for any $f,g \in L^2(\R)$, the scalograms $|W_\Psi f|$ and $|W_\Psi g|$ agree if and only if the following two properties are satisfied:  
	\begin{enumerate}
		\item[(i)] $|\widehat{f}| = |\widehat{g}|$, 
		\item[(ii)] $\forall \lambda \in \Lambda\,:\,|f \ast \psi_\lambda^\ast| = |g \ast \psi_\lambda^\ast|$ and $\big| \widehat{f \ast \psi_\lambda^\ast}\big| = \big| \widehat{g \ast \psi_\lambda^\ast}\big|$. 
	\end{enumerate} 
\end{proposition}
\begin{proof}
	In view of Theorem~\ref{thm: characterization of properties C and P} it suffices to show that $\Psi$ satisfies the spectral injectivity condition \textup{(SI)}; that is, we have to prove that the kernel of the linear operator 
	\[
	M_\Psi : L^1(\R) \to \mathfrak{B}_\nu(\R^\ast), 
	\quad (M_\Psi F)(\lambda) = \int_{\R} F(\xi)\, |\widehat{\psi}(\lambda \xi)|^2 \, \dxi ,
	\]
	is trivial, where $\nu = \frac{d\lambda}{|\lambda|}$ denotes the standard Haar measure on $\R^\ast$. Let $F\in L^1(\R)$ and suppose that 
	\[
	\int_{\R} F(\xi)\, |\widehat{\psi}(\lambda \xi)|^2 \, \dxi = 0 
	\quad \text{for a.e. } \lambda\in \R^\ast.
	\]
	Set $G := |\widehat{\psi}|^2$. Then $G\geq 0$, $G\in L^1(\R)$, and $\supp(G)\subset \R_{>0}$ is compact. For a.e.\ $\lambda>0$ it follows that
	\[
	\int_{\R_{>0}} F(\xi)\, |\widehat{\psi}(\lambda \xi)|^2 \, \dxi = 0 .
	\]
	Make the substitutions $\xi = e^{-r}$ and $\lambda = e^{s}$, and define
	\[
	u(r) := -F(e^{-r})e^{-r}, 
	\qquad v(r) := G(e^{r}) .
	\]
	Then $u,v\in L^1(\R)$, and for a.e.\ $s\in \R$ we obtain
	\[
	(u*v)(s) 
	= -\int_{\R} F(e^{-r})e^{-r}\, G(e^{s-r}) \, dr
	= \int_{\R_{>0}} F(\xi)\, |\widehat{\psi}(\lambda \xi)|^2 \, \dxi 
	= 0.
	\]
	Hence $\widehat{u}\,\widehat{v}=0$. Moreover, $v$ is compactly supported. By the Paley--Wiener theorem, $0\neq \widehat{v}$ admits an analytic extension to the complex plane, and therefore its zero set $Z$ has no accumulation point. Thus $\widehat{u}$ must vanish on $\R\setminus Z$. Since $\widehat{u}$ is continuous, it follows that $\widehat{u}$ vanishes on all of $\R$, and therefore $u=0$. Consequently, $F=0$, which shows that the kernel of $M_\Psi$ is trivial. This completes the proof.
\end{proof}
We now turn to the behavior of our two notions of Cheeger constant. The first result concerns the kernel Cheeger constant. It shows that small Cheeger constants can be obtained by a simple spatial separation argument, namely by shifting components far apart (cf. Example~\ref{ex: kernel Cheeger vs separation by means of shifts}).
\begin{proposition}\label{prop: kernel Cheeger of G+T_zH conv to zero}
	Suppose that $X$ is a non-compact SLCA group, equipped with Haar measure $\mu$ and Borel $\sigma$-algebra $\Sigma$. Moreover, assume that the kernel $k$ is translation-invariant, i.e. for all $x,y,z\in X$, it holds that
	\[k(x+z,y+z)=k(x,y).\]
	Finally, let $0\neq \gf,\hf\in \Hcal$. 
	
	Then, for all $\varepsilon>0$ there exists a compact set $Q\subset X$ such that, for all $z\in X\setminus Q$,
	\[\Cheeger_{\gf+T_z\hf} <\varepsilon.\]
\end{proposition}
\begin{proof}
	First note that, for any \(z\in X\), the translation-invariance of \(k\) ensures that \(\gf+T_z\hf\in\Hcal\).
	
	Without loss of generality, let us assume that $\norm{\gf}_2=\norm{\hf}_2=1$. Recalling Definition~\ref{def: kernel Cheeger constant of f} of the kernel Cheeger constant, it suffices to show that we can choose $Q$ so that, for all $z\in Q^c$, there exists a set $S=S(z)\in \Sigma$ with
	\begin{align}\label{eq: Cheeger quotient}
		\frac{\norm{[K,P_S](\gf+T_z\hf)}_2^2}{\norm{P_S(\gf+T_z\hf)}_2^2 \wedge \norm{P_{S^c}(\gf+T_z\hf)}_2^2 } \leq \varepsilon.
	\end{align}
	Let us introduce a tolerance level $\delta>0$, to be specified at the end of the proof. 
	Now, there exists a ($\delta$-dependent) compact neighborhood $Q_0$ of $0$ (the neutral element of $X$) such that for every $z\in Q_0^c$, 
	\[
	\ip{|\gf|,T_z|\hf|}\leq \delta, \quad \text{as well as} \quad \left(\norm{P_{Q_0}\gf}_2^2\wedge\norm{P_{Q_0}\hf}_2^2\right) \geq 1-\delta.
	\]
	We claim that the compact ($\delta$-dependent) set $Q:=Q_0-Q_0\subset X$ does the job together with $S=Q_0$, as soon as one chooses $\delta$ small enough. 
	
	Fix any $z\in Q^c$. We start by bounding the denominator of the left-hand side of \eqref{eq: Cheeger quotient}. Noting that $Q_0\subset Q$, it follows that
	\begin{align*}
		\norm{P_{Q_0}(\gf+T_z\hf)}_2^2 
		\geq \norm{P_{Q_0}\gf}_2^2-2 \,|\ip{P_{Q_0}\gf,P_{Q_0}T_z\hf}|
		\geq \norm{P_{Q_0}\gf}_2^2-2 \,\ip{|\gf|,T_z|\hf|} \geq 1-3\delta
	\end{align*}
	Now, observe that $P_{Q_0^c}T_z=T_zP_{Q_0^c-z}$ and $Q_0\subset Q_0^c-z$. Therefore, we also have
	\begin{align*}
		\norm{P_{Q_0^c}(\gf+T_z\hf)}_2^2 
		\geq \norm{P_{Q_0^c}T_z\hf}_2^2-2 \delta
		= \norm{P_{Q_0^c-z}\hf}_2^2-2 \delta \geq \norm{P_{Q_0}\hf}_2^2-2 \delta \geq 1-3\delta.
	\end{align*}
	Together, this yields
	\begin{align}\label{eq: prop aux equation 1}
		\left(\norm{P_{Q_0}(\gf+T_z\hf)}_2^2 \wedge \norm{P_{Q_0^c}(\gf+T_z\hf)}_2^2\right)\geq 1-3\delta.
	\end{align}
	Next, we derive an upper bound on the commutator norm. First, note that
	\[\norm{[K,P_{Q_0}](\gf+T_z\hf)}_2^2\leq 2 \left(\norm{[K,P_{Q_0}]\gf}_2^2+\norm{[K,P_{Q_0}]T_z\hf}_2^2\right).\]
	Further, we have
	\[\norm{[K,P_{Q_0}]\gf}_2^2=\norm{[K,P_{Q_0^c}]\gf}_2^2\leq 2\left(\norm{KP_{Q_0^c}\gf}_2^2+\norm{P_{Q_0^c}\gf}_2^2\right)\leq 4\norm{P_{Q_0^c}\gf}_2^2\leq 4\delta.\]
	The translation-invariance of $k$ implies that $K$ and $T_z$ commute. Thus, we have
	\[\norm{[K,P_{Q_0}]T_z\hf}_2^2= \norm{[K,P_{Q_0-z}]\hf}_2^2\leq 4\norm{P_{Q_0-z}\hf}_2^2\leq 4\norm{P_{Q_0^c}\hf}_2^2\leq 4\delta.\]
	Combining the previous estimates shows that 
	\begin{align}\label{eq: prop aux equation 2}
		\norm{[K,P_{Q_0}](\gf+T_z\hf)}_2^2\leq 16\delta.
	\end{align}
	Plugging \eqref{eq: prop aux equation 1} and \eqref{eq: prop aux equation 2} into \eqref{eq: Cheeger quotient} concludes the proof as soon as one chooses $\delta>0$ with \[\frac{16\delta}{1-3\delta}<\varepsilon.\]
	This concludes the proof of the proposition.
\end{proof}

Finally, we supply the proof of a claim from Example~\ref{ex: Cheeger of f+T_x f}. It exhibits a sharp contrast between the behavior of the graph and kernel Cheeger constants under increasing spatial separation.
We use the same notation from Section~\ref{sec: Cheeger const semi-discrete}.

\begin{proposition}\label{prop: aux lem Cheeger const h vs h+T_xh}
	Suppose that $A$ is non-compact, and let $0\neq h\in L^2(A)$. Then the following hold:
	\begin{enumerate}
		\item[(a)] For all $\varepsilon>0$ there exists a compact set $Q_\varepsilon\subset A$ such that, for every $x\in A\setminus Q_\varepsilon$, the \emph{kernel Cheeger constant} satisfies
		\[\Cheeger_{W_\Psi(h+T_xh)} <\varepsilon.\]
		\item[(b)] If the graph $G(h)$ is finite and has at least two vertices, then there exists a compact set $Q^\prime\subset A$ such that, for every $x\in A\setminus Q^\prime$, the \emph{graph Cheeger constant} satisfies
		\[
		\Cheeger_{G(h+T_xh)} \geq \frac{\Cheeger_{G(h)}}{4}>0.
		\]
	\end{enumerate}
\end{proposition}

\begin{proof}
	The proof of part~(a) is analogous to that of Proposition~\ref{prop: kernel Cheeger of G+T_zH conv to zero} and is therefore omitted.
	
	We now prove part~(b).
	Let $E$ and $V$ denote the edge and vertex sets of $G(h)$, respectively.
	Recall that for any $f,g\in L^2(A)$ and any $\varepsilon>0$ there exists a compact set $Q^{\prime\prime}\subset A$ such that for every $x\in A\setminus Q^{\prime\prime}$,
	\[
	\bigl|\ip{f,T_xg}\bigr|\leq \varepsilon.
	\]
	Since translation commutes with convolution and $G(h)$ is finite, we can choose a compact set $Q^\prime\subset A$ such that for every $x\in A\setminus Q^\prime$ and every edge $(\lambda,\lambda')\in E$,
	\[
	2\,\bigl|\ip{h*\psi_\lambda*\psi_{\lambda'},(T_xh)*\psi_\lambda*\psi_{\lambda'}}\bigr|\leq \norm{h*\psi_\lambda*\psi_{\lambda'}}_2^2.
	\]
	Thus,
	\begin{align*}
		\norm{(h+T_xh)*\psi_\lambda*\psi_{\lambda'}}_2^2
		&= \norm{h*\psi_\lambda*\psi_{\lambda'}}_2^2 + \norm{(T_xh)*\psi_\lambda*\psi_{\lambda'}}_2^2 \\
		&\qquad + 2\,\re\ip{h*\psi_\lambda*\psi_{\lambda'},(T_xh)*\psi_\lambda*\psi_{\lambda'}} \\
		&\geq 2\norm{h*\psi_\lambda*\psi_{\lambda'}}_2^2 - 2 \bigl|\ip{h*\psi_\lambda*\psi_{\lambda'},(T_xh)*\psi_\lambda*\psi_{\lambda'}}\bigr| \\
		&\geq \norm{h*\psi_\lambda*\psi_{\lambda'}}_2^2.
	\end{align*}
	On the other hand, for any $\lambda\in V$ we have
	\[
	\norm{(h+T_xh)*\psi_\lambda}_2^2
	\leq 2\bigl(\norm{h*\psi_\lambda}_2^2+\norm{(T_xh)*\psi_\lambda}_2^2\bigr)
	= 4\norm{h*\psi_\lambda}_2^2.
	\]
	Combining these two estimates yields
	\begin{align*}
		\Cheeger_{G(h+T_xh)}
		&= \inf_{\emptyset \neq S \subsetneq V} 
		\frac{\sum_{(\lambda,\lambda') \in \partial S} \norm{(h+T_xh)*\psi_\lambda*\psi_{\lambda'}}_2^2}{\left(\sum_{\lambda \in S} \norm{(h+T_xh)*\psi_\lambda}_2^2\right) \wedge \left(\sum_{\lambda \in S^c} \norm{(h+T_xh)*\psi_\lambda}_2^2 \right)} \\
		&\geq \inf_{\emptyset \neq S \subsetneq V} 
		\frac{\sum_{(\lambda,\lambda') \in \partial S} \norm{h*\psi_\lambda*\psi_{\lambda'}}_2^2}{4\,\left(\bigl(\sum_{\lambda \in S} \norm{h*\psi_\lambda}_2^2\bigr) \wedge \bigl(\sum_{\lambda \in S^c} \norm{h*\psi_\lambda}_2^2\bigr)\right)} \\
		&= \frac{\Cheeger_{G(h)}}{4}.
	\end{align*}
	Since $G(h)$ is finite and connected, we have $\Cheeger_{G(h)}>0$. 
	This completes the proof.
\end{proof}

\section*{Acknowledgement}

The authors acknowledge funding by the Deutsche Forschungsgemeinschaft (DFG, German Research Foundation)---Project number 442047500 through the Collaborative Research Center ``Sparsity and Singular Structures'' (SFB 1481).

	\printbibliography
	
\end{document}